\documentclass[11pt]{amsart}
\usepackage{amssymb}
\usepackage{bbm}
\usepackage{hyperref}
\hypersetup{
    colorlinks=true, 
    linktoc=all,     
    linkcolor=blue,
    citecolor=red,
    filecolor=black,
    urlcolor=blue   
}
\usepackage{enumerate}
\usepackage[inline]{enumitem}
\makeatletter
\newcommand{\inlineitem}[1][]{%
\ifnum\enit@type=\tw@
    {\descriptionlabel{#1}}
  \hspace{\labelsep}%
\else
  \ifnum\enit@type=\z@
       \refstepcounter{\@listctr}\fi
    \quad\@itemlabel\hspace{\labelsep}%
\fi} \makeatother
\newcommand{\ga}{\alpha}
\newcommand{\gb}{\beta}

\newcommand{\gl}{\lambda}

\newcommand{\gp}{\pi}

\newcommand{\gs}{\sigma}

\newcommand{\gt}{\tau}

\newcommand{\Gs}{\Sigma}

\newcommand{\subs}{\subset}
\newcommand{\sups}{\supset}

\newcommand{\sbnq}{\subsetneq}

\newcommand{\bs}{\backslash}

\newcommand{\nin}{\notin}

\newcommand{\ti}{\tilde}

\newcommand{\mbb}{\mathbb}

\newcommand{\mcl}{\mathcal}

\newcommand{\ol}{\overline}

\newcommand{\us}{\underset}
\newcommand{\os}{\overset}
\newcommand{\Lra}{\Leftrightarrow}

\newcommand{\lra}{\longrightarrow}

\newcommand{\Z}{\mbb Z}
\newcommand{\R}{\mcl R}

\newcommand{\Ra}{\Rightarrow}

\newcommand{\es}{\emptyset}

\newcommand{\eqdef}{\overset{\mathrm{def}}{=\joinrel=}}
\newcommand{\equ}[1]{%
\begin{equation*}
#1
\end{equation*}
}
\newcommand{\equa}[1]{%
\begin{equation*}
\begin{aligned}
#1
\end{aligned}
\end{equation*}
}
\newcommand{\equan}[2]{%
\begin{equation}
\label{Eq:#1}
\begin{aligned}
#2
\end{aligned}
\end{equation}
}
\DeclareMathOperator{\Det}{Det}
\newcommand{\mattwo}[4]{%
\begin{pmatrix}
  #1 & #2\\ #3 & #4
\end{pmatrix}
}

\newcommand{\mattwofour}[8]{%
\begin{pmatrix}
  #1 & #2 & #3 & #4\\#5 & #6 & #7 & #8
\end{pmatrix}
}

\newcommand{\matk}[1]{
\begin{pmatrix}
#1_{11}       & #1_{12}     & #1_{13}       &   \cdots  &   #1_{1(k-1)}     & #1_{1k}\\
#1_{21}       & #1_{22}     & #1_{23}       &   \cdots  &   #1_{2(k-1)}     & #1_{2k}\\
#1_{31}       & #1_{32}     & #1_{33}       &   \cdots  &   #1_{3(k-1)}     & #1_{3k}\\
\vdots        & \vdots      & \vdots        &   \ddots  &   \vdots          & \vdots\\
#1_{(k-1)1}   & #1_{(k-1)2} & #1_{(k-1)3}   &   \cdots  &   #1_{(k-1)(k-1)} & #1_{(k-1)k}\\
#1_{k1}       & #1_{k2}     & #1_{k3}       &   \cdots  &   #1_{k(k-1)}     & #1_{kk}\\
\end{pmatrix}
}

\theoremstyle{plain}
\newtheorem{theorem}{Theorem}[section]
\newtheorem{defn}[theorem]{Definition}

\newtheorem{prop}[theorem]{Proposition}
\newtheorem{lemma}[theorem]{Lemma}
\newtheorem{cor}[theorem]{Corollary}

\newtheorem{ques}[theorem]{Question}
\newtheorem{claim}[theorem]{Claim}

\newtheorem{obs}[theorem]{Observation}

\newtheorem{remark}[theorem]{Remark}
\newtheorem{example}[theorem]{Example}
\newtheorem{note}[theorem]{Note}

\numberwithin{equation}{section}
\begin{document}


\title[On the Surjectivity of Certain Maps]{On the Surjectivity of Certain Maps}
\author[C.P. Anil Kumar]{C.P. Anil Kumar}
\address{Stat Math Unit, Indian Statistical Institute, 8th Mile Mysore Road, RVCE Post, Bangalore-560059, India}
\email{akcp1728@gmail.com}


\begin{abstract}
We prove in this article the surjectivity of three maps. We prove in
Theorem~\ref{theorem:CRTSURJMOSTGENCASE} the surjectivity of the
Chinese remainder reduction map associated to the projective space of an
ideal with a given factorization into ideals whose radicals are
pairwise distinct maximal ideals. In
Theorem~\ref{theorem:SurjModIdeal} we prove the surjectivity of the
reduction map of the strong approximation type for a ring quotiented
by an ideal which satisfies unital set condition. In
Theorem~\ref{theorem:SurjSLkHighDimension} we prove for Dedekind type domains which include 
Dedekind domains, for $k \geq 2$, the map from $k\operatorname{-}$dimensional special linear
group to the product of projective spaces of $k\operatorname{-}$mutually co-maximal
ideals associating the $k\operatorname{-}$rows or $k\operatorname{-}$columns is surjective.
Finally this article leads to three interesting
questions~[\ref{question:Subgroups},~\ref{question:Variety},~\ref{question:ProjectiveSpace}]
mentioned in the introduction section.
\end{abstract}
\maketitle


\section{\bf{Introduction and a brief survey}}
This article concerns a generalization of Chinese Remainder Theorem in two different contexts. One is regarding the decomposition of a projective space as a finite product of projective spaces
associated to ideals. Projective spaces and projective varieties are of immense interest to algebraic geometers, as quite a few results can be stated over such spaces more precisely than over
affine spaces and affine varieties.
The other generalization is with regard to the representation of congruence classes of special linear matrices over a ring modulo an ideal of a certain type.
Similar ideas are discussed in the survey article~\cite{SA} by A.S.~Rapinchuk.
This article also concerns a third context about surjectivity of a map about representing elements of  projective spaces associated to a finite set of ideals by a special linear matrix over a  
Dedekind type domain (refer to Definition~\ref{defn:Good}). Moreover here we bring together similar concepts in ring theory (mainly in Section~\ref{sec:Preliminaries}) required to answer these 
three new questions stated in Sections~[\ref{sec:FMR},~\ref{sec:SMR},~\ref{sec:TMR}]. 
The theory of Dedekind domains has been studied from at least three perspectives including algebraic number theory (arithmetic) (R.~Ash~\cite{RA}, A.~Frolich and M.J.~Taylor~\cite{AFMJT}, S.~Lang~\cite{SL} and et al.), 
algebraic geometry (M.~Nagata~\cite{MN1},~\cite{MN2},~\cite{MN3}) and also ring theory. Here in the literature from the perspective of ring theory, the authors 
R.W.~Gilmer~\cite{GRW}, K.A.~Loper~\cite{LKA}, H.~Uda~\cite{UH} and et al. have studied rings which upon imposing noetherian condition gives rise to a Dedekind domain. They are known as almost Dedekind domains. 
Here also we define Dedekind type domain which upon imposing noetherian condition gives rise to a Dedekind domain (refer to Appendix Section~\ref{sec:Appendix}). 
This article is mostly self contained.


\subsection{\bf{The main results and open questions}}
~\\
We begin this section with a few definitions before we state the three main results and three open questions. 
Below in Definitions~[\ref{defn:IdealRadArbitraryIntegers},~\ref{defn:GEQRE}] we define the projective spaces associated to a certain class denoted by 
$\mcl{INTRAD}(\R)^{*}$ of ideals over arbitrary commutative rings $\R$ with unity.


\begin{defn}
\label{defn:IdealRadArbitraryIntegers} Let $\R$ be a commutative ring
with identity. Define the set of non-zero ideals 
\equan{IdealClass}{\mcl{INTRAD}(\R)^{*}=& \{\mcl{I}\sbnq \mcl{R}\mid (0)\neq \mcl{I} 
\text{ is an arbitrary intersection of ideals of } \R\\ 
&\text{whose radicals are all distinct maximal ideals }\}} and
$\mcl{INTRAD}(\R)=\mcl{INTRAD}(\R)^{*} \cup \{(0)\}$. 
\end{defn}


\begin{defn}
\label{defn:GEQRE} Let $\R$ be a commutative ring with identity. Let $k\in \mbb{N}$.
Let $0 \neq \mcl{I}\subs \R$ be an ideal such that $\mcl{I} \in
\mcl{INTRAD}(\R)$. Let $(a_0,a_1,a_2,\ldots,a_k)$,
$(b_0,b_1,b_2,\ldots,b_k) \in \R^{k+1}$. Suppose each of the sets
$\{a_0,a_1,a_2,\ldots,a_k\}$, $\{b_0,b_1,b_2,\ldots,b_k\}$ generate
the unit ideal $\R$. We say \equ{(a_0,a_1,a_2,\ldots,a_k) \sim_{GR}
(b_0,b_1,b_2,\ldots,b_k)} if and only if $a_ib_j-a_jb_i \in \mcl{I}$
for $0 \leq i < j \leq k$. This relation $\sim_{GR}$ is an
equivalence relation (see Lemma~\ref{lemma:GEQRE}). The equivalence
class of $(a_0,a_1,a_2,\ldots,a_k)$ is denoted by
$[a_0:a_1:a_2:\ldots:a_k]$. Define the $k\operatorname{-}$dimensional projective
space corresponding to $\mcl{I}$ denoted by
\equ{\mbb{PF}^k_{\mcl{I}}=\{[a_0:a_1:a_2:\ldots:a_k]\mid\text{ the
set }\{a_0,a_1,a_2,\ldots,a_k\} \subs \R \text{ generates the unit
ideal }=\R\}.} Note here we can have elements
$\{a_0,a_1,a_2,\ldots,a_k\},$ where each $a_i$ is not a unit $\mod
\mcl{I}$.
\end{defn}


In Section~\ref{sec:ProjSpace} we prove the existence of such projective
spaces. Now we define in Definition~\ref{defn:unitalset} below when a finite subset of a commutative ring $\R$ is a unital set.


\begin{defn}
\label{defn:unitalset}
Let $\R$ be a commutative ring with unity. Let $k\in \mbb{N}$. We say a finite subset 
\equ{\{a_1,a_2,\ldots,a_k\}\subs \R} consisting of $k\operatorname{-}$elements (possibly with repetition) 
is unital or a unital set if the ideal generated by the elements of the set is the unit ideal.
\end{defn}


Based on the previous definition, we make a relevant definition, the unital set condition for an ideal.
\begin{defn}[Unital set condition for an ideal]
\label{defn:UnitalSetCond} Let $\R$ be a commutative ring with unity. Let $k\in \mbb{N}$.
Let $\mcl{I} \sbnq \mcl{\R}$ be an ideal. We say $\mcl{I}$ satisfies
unital set condition $USC$ if for every unital set
$\{a_1,a_2,\ldots,a_k\} \subs \R$ with $k \geq 2$, there exists an
element $j \in (a_2,\ldots,a_k)$ such that $a_1+j$ is a unit modulo
$\mcl{I}$.
\end{defn}


Now we mention one more definition of a certain type of a ring known as Dedekind type domain.
\begin{defn}
\label{defn:Good} Let $\R$ be a commutative ring with unity. Suppose
the ring $\R$ satisfies the following four properties.
\begin{itemize}
\item (Property 1): For each maximal ideal $\mcl{M}$ we have $\mcl{M}^i \neq
\mcl{M}^{i+1}$ for all $i \geq 0$.
\item (Property 2): $\us{i \geq 0}{\bigcap}\mcl{M}^i=(0)$.
\item (Property 3): $dim_{\frac{\R}{\mcl{M}}}(\frac{\mcl{M}^i}{\mcl{M}^{i+1}})=1$ for all $i \geq 0$.
\end{itemize}
So as a consequence of these properties for a noetherian ring $\R$, it also satisfies the 
following property (see Appendix Section~\ref{sec:Appendix}).
\begin{itemize}
\item (Property 4): Every non-zero element $r\in \R$ is contained in finitely many
maximal ideals.
\end{itemize}
We say a ring $\R$ is a Dedekind type domain if it is a field or if it satisfies properties 
$(1),(2),(3),(4)$. A commutative ring with unity which satisfies these properties is actually a domain.
The examples for Dedekind type domains include integers, principal ideal domains, discrete valuations rings, 
Dedekind domains, Dedekind domains which are obtained as the localizations at any multiplicatively closed set 
of a Dedekind domain. We remark that a noetherian Dedekind type domain is a Dedekind domain (see Appendix Section~\ref{sec:Appendix}).
\end{defn}


Now we state the main results and
open questions.
\subsubsection{\bf{The first main result}}
\label{sec:FMR}
~\\
The first main result concerns the surjectivity of the Chinese remainder 
reduction map associated to a projective space of an ideal (refer to
Definition~\ref{defn:GEQRE}) with a given co-maximal
ideal factorization which is stated as:
\begin{theorem}
\label{theorem:CRTSURJMOSTGENCASE} Let $\R$ be a commutative ring
with unity. Let $k\in \mbb{N}$. Let $\mcl{I}=\mcl{Q}_1\mcl{Q}_2\ldots\mcl{Q}_k$ be a non-zero ideal, 
where $rad(\mcl{Q}_i)=\mcl{M}_i,1\leq i\leq k$ are pairwise distinct maximal ideals in
$\R$. Then the Chinese remainder reduction map associated to the
projective space \equ{\mbb{PF}^{l+1}_{\mcl{I}} \lra
\mbb{PF}^{l+1}_{\mcl{Q}_1} \times \mbb{PF}^{l+1}_{\mcl{Q}_2} \times
\ldots \times \mbb{PF}^{l+1}_{\mcl{Q}_k}} is surjective (in fact bijective).
\end{theorem}
We also give counter Example~\ref{example:CounterExample}
where the surjectivity does not hold in the case of projective
spaces associated to a product of two prime ideals each of which
cannot be expressed as a finite intersection of ideals whose
radicals are pairwise distinct maximal ideals.


\subsubsection{\bf{The second main result}}
\label{sec:SMR}
~\\
The second main result is a result of strong approximation type.
Here we give a criterion called the $USC$ which is
given in Definition~\ref{defn:UnitalSetCond} and prove the following
surjectivity theorem which is stated as:
\begin{theorem}
\label{theorem:SurjModIdeal} Let $\R$ be a commutative ring with
unity. Let $k\in \mbb{N}$. Let \equ{SL_k(\R)=\{A\in M_{k\times k}(\R) \mid \Det(A)=1\}}
Let $\mcl{I} \sbnq \R$ be an ideal which satisfies the unital
set condition (see Definition~\ref{defn:UnitalSetCond}). Then the reduction map
\equ{SL_k(\R) \lra SL_k(\frac{\R}{\mcl{I}})} is surjective.
\end{theorem}
A survey of results on strong approximation can be found
in~\cite{SA}. 


\subsubsection{\bf{The third main result}}
\label{sec:TMR}
~\\
The third main result concerns the surjectivity of
another map from the group $SL_k(\R)$ to a product of $k\operatorname{-}$projective
spaces associated to $k\operatorname{-}$pairwise co-maximal ideals. 
The theorem is stated as:
\begin{theorem}
\label{theorem:SurjSLkHighDimension} Let $\R$ be a commutative ring
with unity. Suppose $\R$ is a Dedekind type domain (refer to
Definition~\ref{defn:Good}). Let $\mathit{M}(\R)$ be the monoid
generated by maximal ideals in $\R$ (refer to Section~\ref{sec:UFMaximalMonoid}). 
Let $k \geq 2$ be a positive integer.
Let $\mcl{I}_1,\mcl{I}_2,\ldots,\mcl{I}_k\in \mathit{M}(\R)$ be $k\operatorname{-}$
pairwise co-maximal proper ideals. Consider 
\equ{SL_k(\R)=\{A=[a_{ij}]_{k \times k} \in M_{k\times k}(\R)
\mid \Det(A)=1\}.} Then the maps \equ{\gs_1,\gs_2: SL_{k}(\R) \lra
\mbb{PF}^{k-1}_{\mcl{I}_1} \times \mbb{PF}^{k-1}_{\mcl{I}_2} \times
\ldots \times \mbb{PF}^{k-1}_{\mcl{I}_k}} given by 
\equa{
\gs_1:(A) &=
([a_{11}:a_{12}:\ldots:a_{1k}],[a_{21}:a_{22}:\ldots:a_{2k}],\ldots,
[a_{k1}:a_{k2}:\ldots:a_{kk}]),\\
\gs_2:(A) &=
([a_{11}:a_{21}:\ldots:a_{k1}],[a_{12}:a_{22}:\ldots:a_{k2}],\ldots,
[a_{1k}:a_{2k}:\ldots:a_{kk}])
}
are surjective.
\end{theorem}


This article leads to the following three open questions.
\subsubsection{\bf{The first open question}}
~\\
The first question is concerning surjectivity of the map for subgroups of $SL_k(\R)$. It is stated as:
\begin{ques}
\label{question:Subgroups} Let $\R$ be a commutative ring with unity.
Suppose $\R$ is a Dedekind type domain (refer to Definition~\ref{defn:Good}).
Let $\mathit{M}(\R)$ be the monoid generated by maximal ideals in
$\R$ (refer to Section~\ref{sec:UFMaximalMonoid}). Let $k \geq 2$ be a positive
integer. Let $\mcl{I}_1,\mcl{I}_2,\ldots,\mcl{I}_k\in \mathit{M}(\R)$ be
$k\operatorname{-}$ pairwise co-maximal proper ideals. Let $G_k(\R) \subs SL_k(\R)$ be a subgroup. Under what
conditions on $G_k(\R)$ are the maps \equ{\gs_1,\gs_2: G_k(\R) \lra
\mbb{PF}^{k-1}_{\mcl{I}_1} \times \mbb{PF}^{k-1}_{\mcl{I}_2} \times
\ldots \times \mbb{PF}^{k-1}_{\mcl{I}_k}} given by 
\equa{
\gs_1:(A)&=
([a_{11}:a_{12}:\ldots:a_{1k}],[a_{21}:a_{22}:\ldots:a_{2k}],\ldots,
[a_{k1}:a_{k2}:\ldots:a_{kk}]),\\
\gs_2:(A) &=
([a_{11}:a_{21}:\ldots:a_{k1}],[a_{12}:a_{22}:\ldots:a_{k2}],\ldots,
[a_{1k}:a_{2k}:\ldots:a_{kk}])
}
surjective?
\end{ques}


\subsubsection{\bf{The second open question}}
~\\
The second question is concerning surjectivity of the map, where the
equation is different from the defining equation of $SL_k(\R) \subs
M_{k \times k}(\R)$. Before stating the following open question, we
mention that we prove another surjectivity
Theorem~\ref{theorem:SurjSumProductEquation} for the sum-product
equation (refer to equation~\ref{Eq:SumProduct}) in Section~\ref{sec:SumProductEquation}. 
Now we state the question concerning general varieties in a slightly general context:
\begin{ques}
\label{question:Variety} Let $\R$ be a commutative ring with unity.
Suppose $\R$ is a Dedekind type domain (refer to Definition~\ref{defn:Good}).
Let $\mathit{M}(\R)$ be the monoid generated by maximal ideals in
$\R$. Let $k \geq 2$ be a positive
integer. Let $\mcl{I}_1,\mcl{I}_2,\ldots,\mcl{I}_k\in \mathit{M}(\R)$ be
$k\operatorname{-}$ pairwise co-maximal proper ideals.  Let $M_{k\times k}(\R)$ be the set of $k \times k$ matrices
with entries in $\R$. Let $f:M_{k\times k}(\R) \lra \R$ be a polynomial
function in the entries. Suppose $f(g=[g_{ij}]_{k \times k})=0$
implies each row of $g$ is unital. Let $V_f(\R)=\{x=[x_{ij}] \in
M_{k\times k}(\R)\mid \text{ such that
}f(x_{11},x_{12},\ldots,x_{kk})=0\}$. For what type of equations $f=0$ satisfying such a condition, is
the map \equ{\gs_1: V_f(\R) \lra \mbb{PF}^{k-1}_{\mcl{I}_1} \times
\mbb{PF}^{k-1}_{\mcl{I}_2} \times \ldots \times
\mbb{PF}^{k-1}_{\mcl{I}_k}} given by 
\equ{
\gs_1:(A) =
([a_{11}:a_{12}:\ldots:a_{1k}],[a_{21}:a_{22}:\ldots:a_{2k}],\ldots,
[a_{k1}:a_{k2}:\ldots:a_{kk}])
} surjective?
\end{ques}


\subsubsection{\bf{The third open question}}
~\\
The third question is the following.
\begin{ques}
\label{question:ProjectiveSpace}
Classify geometrically defined spaces which are actually full projective
spaces associated to an ideal of a ring.
\end{ques}
Here below we remark on the projective space associated to an ideal as an
application of Chinese remainder reduction isomorphism.
\begin{remark}
\label{remark:WhatSpaces}
This remark concerns the question as to what spaces can be
considered as projective spaces associated to ideals.  We motivate this via some examples.

Let $\mbb{K}$ be an algebraically closed field. Then we know,
via segre embedding, the space $(\mbb{PF}^k_{\mbb{C}})^n=\mbb{PF}^k_{\mbb{C}} \times \mbb{PF}^k_{\mbb{C}} \times \ldots \times \mbb{PF}^k_{\mbb{C}}$
is a projective algebraic variety in a suitable high dimensional
projective space. However it is also a projective space associated to
an ideal. Suppose $\R$ is a commutative ring with unity and
$\mcl{M}_1,\mcl{M}_2,\ldots,\mcl{M}_n$ are distinct ideals all whose
quotients are isomorphic to $\mbb{C}$ then $(\mbb{PF}^k_{\mbb{C}})^n
= \mbb{PF}^k_{\mcl{I}},$ where
$\mcl{I}=\us{i=1}{\os{n}{\prod}}\mcl{M}_i$ via Chinese remainder reduction
isomorphism. In some cases the fields need not be the same. If
$\mbb{K}_1,\mbb{K}_2,\ldots,\mbb{K}_r$ are $r\operatorname{-}$fields and if
$\mcl{M}_1,\mcl{M}_2,\ldots,\mcl{M}_r$ are pairwise co-maximal ideals
in $\R$ with $\frac {\R}{\mcl{M}_i}=\mbb{K}_i$ then
$\us{i=1}{\os{r}{\prod}}\mbb{PF}^k_{\mbb{K}_i}\cong
\mbb{PF}^k_{\mcl{J}},$ where
$\mcl{J}=\us{i=1}{\os{r}{\prod}}\mcl{M}_i$ via Chinese remainder reduction
isomorphism. For example \equ{\mbb{PF}^2_{\mbb{R}} \times
\mbb{PF}^2_{\mbb{C}} \cong \mbb{PF}^2_{(x(x^2+1))},} where
$\R=\mbb{R}[x],\mcl{M}_1=(x),\mcl{M}_2=(x^2+1)$.
\end{remark}


\section{\bf{Preliminaries}}
\label{sec:Preliminaries}
~\\
This section contains preliminaries which are needed in the proofs of the main results.


\subsection{\bf{Projective spaces associated to ideals in arbitrary commutative rings with
identity}} \label{sec:ProjSpace} 

In this section we prove the well-definedness and existence of $k\operatorname{-}$dimensional
projective spaces.
\begin{lemma}
\label{lemma:GEQRE} Using the notation in
Definition~\ref{defn:GEQRE}, the relation $\sim_{GR}$ is an
equivalence relation.
\end{lemma}
\begin{proof}
The relation is reflexive and symmetric. We need to prove transitivity. Suppose
$(a_0,a_1,a_2,\ldots,a_k)$,$(b_0,b_1,b_2,\ldots,b_k)$,$(c_0,c_1,c_2,\ldots,c_k)
\in \R^{k+1}$ and each of the sets \linebreak
$\{a_0,a_1,a_2,\ldots,a_k\}$,$\{b_0,b_1,b_2,\ldots,b_k\}$,$\{c_0,c_1,c_2,\ldots,c_k\}$
generate the unit ideal $\R$. First consider the case when $0\neq \mcl{I}
\in \mcl{INTRAD}(\R)$ is an ideal whose radical is a maximal ideal
$\mcl{M}$. Suppose $(a_i:0 \leq i \leq k)\sim_{GR}(b_i:0 \leq i \leq k), 
(a_i:0 \leq i \leq k)\sim_{GR}(c_i:0 \leq i \leq k)$. Suppose
without loss of generality $a_1 \nin \mcl{M}$. So $a_1$ is a unit
$\mod \mcl{I}$. We assume $a_1=1$. Now for any $0 \leq i < j \leq k$
we have $b_ic_j=a_1b_ic_j \equiv b_1a_ic_j \equiv
b_1c_ia_j=b_1a_jc_i\equiv a_1b_jc_i=b_jc_i \mod \mcl{I}$. Hence the
transitivity follows for $\mcl{I}$. Since every ideal $0\neq \mcl{I} \in
\mcl{INTRAD}(\R)$ is an intersection of ideals with distinct radical
maximal ideals, Lemma~\ref{lemma:GEQRE} follows for any nonzero
ideal $\mcl{I} \in \mcl{INTRAD}(\R)$.
\end{proof}
This proves the existence and if $\frac {\R}{\mcl{I}}$ is a finite ring then the space
$\mbb{PF}^k_{\mcl{I}}$ is a finite projective space.


\subsection{\bf{On arithmetic progressions}}
\label{sec:FundLemma}
In this section we prove a very useful lemma on arithmetic
progressions for integers, Dedekind type domains and a proposition in the context of schemes.
Remark~\ref{remark:FundLemma} below summarizes these two 
Lemmas~[\ref{lemma:FundLemmaIntegers},~\ref{lemma:FundLemma}] and 
Proposition~\ref{prop:FundLemmaSchemes} in this section. 

\begin{lemma}[A lemma on arithmetic progressions for integers]
\label{lemma:FundLemmaIntegers}
~\\
Let $a,b \in \Z$ be integers with $(a)+(b)=1$. Consider the set
$\{a+nb \mid n \in \mbb{Z}\}$. Let $m \in \mbb{Z}$ be any non-zero
integer. Then there exists an $n_0 \in \mbb{Z}$ and an element of
the form $a+n_0b$ such that $gcd(a+n_0b,m)=1$.
\end{lemma}
\begin{proof}
Assume $a,b$ are both non-zero. Otherwise 
Lemma~\ref{lemma:FundLemmaIntegers} is trivial. Let
$q_1,q_2,q_3,\ldots,q_t$ be the distinct prime factors of $m$.
Suppose $q \mid gcd(m,b)$ then $q \nmid a+nb$ for all $n \in
\mbb{Z}$. Such prime factors $q$ need not be considered. Let $q \mid
m, q \nmid b$. Then there exists $t_q \in \mbb{Z}$ such that the
exact set of elements in the given arithmetic progression divisible
by $q$ is given by \equ{\ldots, a+(t_q-2q)b, a+(t_q-q)b, a+t_qb,
a+(t_q+q)b,a+(t_q+2q)b \ldots} Since there are finitely many such
prime factors for $m$ which do not divide $b$ we get a set of
congruence conditions for the multiples of $b$ as $n \equiv t_q \mod\ q$. In order to get an $n_0$ we solve a different set of
congruence conditions for each such prime factor say for example $n
\equiv t_q+1 \mod\ q$. By Chinese Remainder Theorem we have such
solutions $n_0$ for $n$ which therefore satisfy $gcd(a+n_0b,m)=1$.
\end{proof}

\begin{lemma}[A lemma on arithmetic progressions for Dedekind type domains]
\label{lemma:FundLemma} Let $\mcl{O}$ be a Dedekind type domain. Let
$a,b \in \mcl{O}$ such that sum of the ideals $(a)+(b)=\mcl{O}$.
Consider the set $\mcl{A}=\{a+nb \mid n \in \mcl{O}\}$. Let $\mcl{M}
\subs \mcl{O}$ be any nonzero ideal. Then there exists an $n_0 \in
\mcl{O}$ and an element $a+n_0b \in \mcl{A}$ such that the sum of
the ideals $(a+n_0b)+\mcl{M}=\mcl{O}$.
\end{lemma}
\begin{proof}
Assume $a,b$ are both non-zero as otherwise 
Lemma~\ref{lemma:FundLemma} is trivial. Let the ideal
$\mcl{M}$ be contained in finitely many maximal ideals 
$\mcl{Q}_1,\mcl{Q}_2,\ldots \mcl{Q}_t$. Suppose $\mcl{Q} \in
\{\mcl{Q}_1,\mcl{Q}_2,\ldots,\mcl{Q}_t\}$ and $\mcl{Q} \sups
\mcl{M}+(b)$ then $a+nb \nin \mcl{Q}$ for all $n \in \mcl{O}$
because otherwise both $a,b \in \mcl{Q}$ which is a contradiction.
Such prime ideals $\mcl{Q}$ need not be considered.

Let $\mcl{M} \subs \mcl{Q}$ and  $b \nin \mcl{Q}$ then there exists
$t_{\mcl{Q}} \in \mcl{O}$ such that \equ{\{t \mid a+tb \in
\mcl{Q}\}=t_\mcl{Q}+\mcl{Q}} an arithmetic progression. This can be
proved as follows. Since $b \nin \mcl{Q}$ we have
$(b)+\mcl{Q}=\mcl{O}$. So there exists $t_{\mcl{Q}}$ such that
$a+t_{\mcl{Q}}b \in \mcl{Q}$. If $a+tb \in \mcl{Q}$ then
$(t-t_{\mcl{Q}})b \in \mcl{Q}$. So $t \in t_{\mcl{Q}}+\mcl{Q}$.

Since there are finitely many such maximal ideals $\mcl{Q}$ containing $\mcl{M}$ 
such that $b \nin \mcl{Q}$ we get a set
of congruence conditions for the multiples of $b$ as $n \equiv
t_{\mcl{Q}} \mod\ \mcl{Q}$. In order to get an $n_0$ we solve a
different set of congruence conditions for each such maximal ideal
say for example $n \equiv t_{\mcl{Q}}+1 \mod\ \mcl{Q}$. By
Chinese Remainder Theorem we have such solutions $n_0$ for $n$ which
therefore satisfy $a+n_0b \nin \mcl{Q}$ for all maximal ideals 
$\mcl{Q} \in  \{\mcl{Q}_1,\mcl{Q}_2,\ldots,\mcl{Q}_t\}$ and hence
the sum of the ideals $(a+n_0b)+\mcl{M}=\mcl{O}$.

This proves the Lemma~\ref{lemma:FundLemma} on
arithmetic progressions for Dedekind type domains.
\end{proof}

\begin{prop}
\label{prop:FundLemmaSchemes} Let $X$ be a scheme. Let $Y \subs
X$ be an affine sub-scheme. Let $f,g \in \mcl{O}(Y)$ be two regular
functions on $Y$ such that the unit regular function
$\mathbbm{1}_{Y} \in (f,g)\subs \mcl{O}(Y)$. Let $E\subs Y$ be any
finite set of closed points. Then there exists a regular function $a
\in \mcl{O}(Y)$ such that $f+ag$ is a non-zero element in the
residue field
$k(\mcl{M})=\frac{\mcl{O}(Y)_{\mcl{M}}}{\mcl{M}_{\mcl{M}}}=\frac{\mcl{O}(Y)}{\mcl{M}}$
at every $\mcl{M}\in E$.
\end{prop}
\begin{proof}
Let the set of closed points be given by
$E=\{\mcl{M}_1,\mcl{M}_2,\ldots,\mcl{M}_t\}$. If $g$ vanishes in the
residue field at $\mcl{M}_i$ then for all regular functions $a\in
\mcl{O}(Y),f+ag$ does not vanish in the residue field at
$\mcl{M}_i$. Otherwise both $f,g \in \mcl{M}_i$ which is a
contradiction to $\mathbbm{1}_{Y} \in (f,g)$.

Now consider the finitely many maximal ideals $\mcl{M}\in E$ such
that $g \nin \mcl{M}$. Then there exists $t_{\mcl{M}}$ such that the
set \equ{\{t \mid f+tg \in \mcl{M}\}=t_{\mcl{M}}+\mcl{M}} a complete
arithmetic progression. This can be proved as follows. Since 
$g \nin \mcl{M}$ we have $(g)+\mcl{M}=(\mathbbm{1}_{Y})$. So
there exists $t_{\mcl{M}}$ such that $f+t_{\mcl{M}}g \in \mcl{M}$.
Now if $f+tg \in \mcl{M}$ then $(t-t_{\mcl{M}})g \in \mcl{M}$. Hence
$t \in t_{\mcl{M}}+\mcl{M}$.

Since there are finitely such maximal ideals $\mcl{M}$ such that $g
\nin \mcl{M}$ in the set $E$ we get a finite set of congruence
conditions for the multiples $a$ of $g$ as $a \equiv t_{\mcl{M}}
\mod\ \mcl{M}$. In order to get an $a_0$ we solve a different set of
congruence conditions for each such maximal ideal in $E$ say for
example $a \equiv t_{\mcl{M}}+1 \mod\ \mcl{M}$. By Chinese Remainder
Theorem we have such solutions $a_0$ for $a$ which therefore satisfy
$f+a_0g \nin \mcl{M}$ for all maximal ideals $\mcl{M} \in E$ and
hence the regular function $f+a_0g$ does not vanish in the residue
field $k(\mcl{M})$ for every $\mcl{M} \in E$. This proves 
Proposition~\ref{prop:FundLemmaSchemes}.
\end{proof}
\begin{remark}
\label{remark:FundLemma}
If $a,b\in \Z,gcd(a,b)=1$ then there exist $x,y\in \Z$ such that $ax+by=1$.
Here we note that in general $x$ need not be one unless $a\equiv 1\mod\ b$.
However for any non-zero integer $m$ we can always choose $x=1$ to find
an integer $a+by$ such that $gcd(a+by,m)=1$. In the context of schemes this observation
gives rise to regular functions which do vanish at a given finite set of closed points.
\end{remark}


\subsection{\bf{Ideal avoidance}}
\label{sec:IdealAvoidance}
In this section first we prove below the order prescription
Lemma~\ref{lemma:OrderPrescriptionLemma} before stating Proposition~\ref{prop:IdealAvoidance}
on ideal avoidance. Remark~\ref{remark:functionorderofvanishing} below summarizes
Lemma~\ref{lemma:OrderPrescriptionLemma} and Proposition~\ref{prop:IdealAvoidance}.
\begin{lemma}[Order prescription lemma]
\label{lemma:OrderPrescriptionLemma} Let $\R$ be a commutative ring
with unity. Let $\mcl{I}\subs \R$ be an ideal. Let $\{\mcl{M}_i:1 \leq i \leq t\}$ be a finite non-empty set of
maximal ideals. For each $1 \leq i \leq t$ let $\mcl{M}_i^{m_i}
\sups \mcl{I}$ but $\mcl{M}_i^{m_i+1} \nsupseteq \mcl{I}$ for some non-negative integer $m_i$. 
Then there exists a function $f \in \mcl{I}$ such that $f\in \mcl{I}\bs
\us{i=1}{\os{t}{\bigcup}} \mcl{I}\mcl{M}_i$. In particular \equ{f
\in \bigg(\mcl{M}_i^{m_i}\bs \mcl{M}_i^{m_i+1}\bigg)\cap \mcl{I}\text{ for }1
\leq i \leq t.}
\end{lemma}
\begin{proof}
Let $\mcl{M}_1,\mcl{M}_2,\ldots,\mcl{M}_r$ be the finite set of
maximal ideals for which $m_i=0$ and let
$\mcl{M}_{r+1},\mcl{M}_{r+2},\ldots,\mcl{M}_t$ be the remaining
ideals for which $m_i>0$. So first we observe that for $1 \leq j
\leq r,\mcl{M}_j$ does not contain $\mcl{I}\bigg(\us{i=1,i\neq
j}{\os{t}{\prod}} \mcl{M}_i\bigg)$. So there exists $g_j\in
\mcl{I}\bigg(\us{i=1,i\neq j}{\os{t}{\prod}} \mcl{M}_i\bigg)$ with
$g_j \nin \mcl{M}_j$. Then $g=\us{i=1}{\os{r}{\sum}}g_i \in
\mcl{I},g \nin \mcl{M}_j$ for $j=1,2,\ldots,r$. Let $f_i \in
\mcl{I}\bs \mcl{M}_i^{m_i+1}$ for $i \geq (r+1)$. Let $f_{ij} \in
\mcl{M}_j\bs \mcl{M}_i$. Then we observe that \equ{f=g+\us{i>r,g \in
\mcl{M}_i^{m_i+1}}{\sum}\bigg(f_i\us{j\neq i}{\prod}f_{ij}^{m_j+1}\bigg)\in
\bigg(\mcl{I}\us{i=1}{\os{t}{\bigcap}}(\mcl{M}_i^{m_i}\bs\mcl{M}_i^{m_i+1})\bigg)\bs
\bigg(\us{i=1}{\os{t}{\bigcup}}
\mcl{I}\mcl{M}_i\bigg)} Taking this $f$, 
Lemma~\ref{lemma:OrderPrescriptionLemma} follows.
\end{proof}
\begin{prop}[Ideal avoidance]
\label{prop:IdealAvoidance}
~\\
Let $\R$ be a commutative ring with unity. Suppose for every maximal
ideal $\mcl{M}$, $\us{i=1}{\os{\infty}{\bigcap}} \mcl{M}^i=(0)$. Let
$\mcl{I}\subs \R$ be an ideal. For $r\in \mbb{N}$, let
$\mcl{J}_1,\mcl{J}_2,\ldots,\mcl{J}_r \subs \R$ be $r$ proper ideals such that
\equ{\mcl{I}=\us{i=1}{\os{r}{\bigcup}} \mcl{I}\mcl{J}_i.} Then
$\mcl{I}=(0)$.
\end{prop}
\begin{proof}
Replace the set of ideals $\{\mcl{J}_i:1 \leq i \leq r\}$ by a
finite set of  maximal ideals $\{\mcl{M}_i:1 \leq i \leq s\}$ such
that each maximal ideal $\mcl{M}_i$ contains some ideal $\mcl{J}_j$
for some $j$ and for any ideal $\mcl{J}_i$ there exists a maximal
ideal $\mcl{M}_j$ such that $\mcl{M}_j \sups \mcl{J}_i$. Then we
have \equ{\mcl{I}=\us{i=1}{\os{s}{\bigcup}} \mcl{I}\mcl{M}_i.} Before
applying order prescription Lemma~\ref{lemma:OrderPrescriptionLemma}
for the ideal $\mcl{I}$, if it is non-zero, we observe that a
suitable choice of $m_i$ for $\mcl{M}_i$ exists because of the
hypothesis about intersection property. So $\mcl{I}=(0)$. This
proves Proposition~\ref{prop:IdealAvoidance}.
\end{proof}
\begin{remark}
\label{remark:functionorderofvanishing}
Lemma~\ref{lemma:OrderPrescriptionLemma} gives the existence of functions in an ideal
of functions with a certain order of vanishing at a finite set of closed points
in the context of schemes.   
\end{remark}


\subsection{\bf{The unital lemma}}
\label{sec:UnitalLemma}
In this section we prove unital Lemma~\ref{lemma:Unital} which is useful to obtain a
unit in a $k\operatorname{-}$row unital vector via an $SL_k(\R)\operatorname{-}$elementary
transformation.
\begin{prop}
\label{prop:Unital} Let $\R$ be a commutative ring with unity. Let
$k \geq 2$ be a positive integer. Let $\{a_1,a_2,\ldots,a_k\} \subs
\R$ be a unital set i.e. $\us{i=1}{\os{k}{\sum}}(a_i)=\R$. Let
$\mcl{J}\sbnq \R$ be an ideal contained in only finitely many maximal
ideals. Then there exists $a \in (a_2,\ldots,a_k)$ such that
$a_1+a$ is a unit mod $\mcl{J}$.
\end{prop}
\begin{proof}
Let $\{\mcl{M}_i:1\leq i\leq t\}$ be the finite set of maximal
ideals containing $\mcl{J}$. For example $\mcl{J}$ could be a
product of maximal ideals. Since the set $\{a_1,a_2,\ldots,a_k\}$ is
unital there exists $d \in (a_2,a_3,\ldots,a_k)$ such that
$(a_1)+(d)=(1)$. Now we apply Proposition~\ref{prop:FundLemmaSchemes}, where
$X=Y=Spec(R)$, $E=\{\mcl{M}_i:1\leq i\leq t\}$ to conclude that
there exists $n_0 \in \R$ such that $a=n_0d$ and $a_1+a=a_1+n_0d \nin
\mcl{M}_i$ for $1 \leq i\leq t$. This proves 
Proposition~\ref{prop:Unital}.
\end{proof}

\begin{lemma}
\label{lemma:Unital} Let $\R$ be a commutative ring with unity. Let
$k \geq 2$ be a positive integer. Let $\{a_1,a_2,\ldots,a_k\} \subs
\R$ be a unital set i.e. $\us{i=1}{\os{k}{\sum}}(a_i)=\R$. Let
$E$ be a finite set of maximal ideals in $\R$. Then there exists
$a \in (a_2,\ldots,a_k)$ such that $a_1+a \nin
\mcl{M}$ for all $\mcl{M} \in E$.
\end{lemma}
\begin{proof}
The proof is essentially similar to Proposition~\ref{prop:Unital} even though we need not have to construct an ideal
$\mcl{J}$ which is contained in exactly the maximal ideals in the set $E$.
\end{proof}


\subsection{\bf{Unique factorization maximal ideal monoid of the ring}}
\label{sec:UFMaximalMonoid}
In this section we define in Definition~\ref{defn:Monoid}, the unique factorization monoid of maximal
ideals of the ring. We start by proving below a proposition for a commutative ring
$\R$ which is not a field i.e. ideal $(0)$ is not maximal.
\begin{prop}[Unique factorization]
\label{prop:UniqueFactorization} Let $\R$ be a commutative
ring with unity. Suppose for all maximal ideals $\mcl{M}\subs \R, \mcl{M}^i \neq
\mcl{M}^{i+1}$ for all $i \geq 0$. Let
$I=\mcl{M}_1^{t_1}\mcl{M}_2^{t_2}\ldots\mcl{M}_k^{t_k}=\mcl{N}_1^{s_1}\mcl{N}_2^{s_2}
\ldots\mcl{N}_r^{s_r}$ be two factorizations as a product of powers of distinct
maximal ideals. Then $r=k,\mcl{M}_i=\mcl{N}_i$ with a rearrangement if needed and $t_i=s_i$ for all
$1 \leq i \leq k$.
\end{prop}
\begin{proof}
If $\mcl{M}\sups \mcl{I}$ then $\mcl{M}=\mcl{M}_j$ for some $1 \leq
j \leq k$. So
\equ{\{\mcl{M}_1,\mcl{M}_2,\ldots,\mcl{M}_k\}=\{\mcl{N}_1,\mcl{N}_2,\ldots,\mcl{N}_r\},k=r.}
Now if the ideal $\mcl{I}$ is a power of a maximal ideal then the
power is uniquely determined because $\mcl{M}^i \neq \mcl{M}^{i+1}$
for all $i \geq 0$ and all maximal ideals $\mcl{M} \subs \R$. 
We prove the following respective powers are equal in the following two claims.
\begin{claim}
If $\mcl{M}$ is a maximal ideal and $S=\R\bs \mcl{M}$. Then we have
\equ{S^{-1}\mcl{M}^i=(S^{-1}\mcl{M})^i=\{\frac as\mid a \in
\mcl{M}^i,s \nin \mcl{M}\}} Conversely if $\frac bt \in
S^{-1}\mcl{M}^i$ then $b \in \mcl{M}^i$. Also \equ{S^{-1}\mcl{M}^i
\neq S^{-1}\mcl{M}^{i+1} \text{ for all }i\geq 0.}
\end{claim}
\begin{proof}[Proof of Claim]
Suppose $\frac bt \in S^{-1}\mcl{M}^i$ then there exists $a \in
\mcl{M}^i,s,u \in S$ such that $atu=bsu$. So $b \in \mcl{M}^i$ as
$su \nin \mcl{M}$. Also we have $S^{-1}\mcl{M}^i=(S^{-1}\mcl{M})^i$.
Since $\mcl{M}^i \neq \mcl{M}^{i+1}$ the other inequality of sets in
the claim follows.
\end{proof}
\begin{claim}
If $\mcl{I}=\mcl{M}_1^{t_1}\mcl{M}_2^{t_2}\ldots\mcl{M}_k^{t_k}$ and
$S=\mcl{R}\bs \mcl{M}_1$ then $S^{-1}\mcl{I}=S^{-1}\mcl{M}_1^{t_1}$.
\end{claim}
\begin{proof}[Proof of Claim]
Let $\frac bt \in S^{-1}\mcl{I}$ with $b \in \mcl{I},s \in S$. Then
$b=\us{j=1}{\os{k}{\sum}} b_jc_j$ with $b_j \in \mcl{M}_1^{t_1},c_j
\in \mcl{M}_2^{t_2}\ldots\mcl{M}_k^{t_k}$. So $\frac bt \in
S^{-1}\mcl{M}_1^{t_1}$. Conversely if $b \in \mcl{M}_1^{t_1}$ then
pick $s_i\in \mcl{M}_i\bs \mcl{M}_1,2 \leq i \leq k$ then for any
$\frac bs \in S^{-1}\mcl{M}_1^{t_1},\frac bs =\frac
{bs_2^{t_2}s_3^{t_3}\ldots s_k^{t_k}}{ss_2^{t_2}s_3^{t_3}\ldots
s_k^{t_k}} \in S^{-1}\mcl{I}$. So
$S^{-1}\mcl{M}^{t_1}=S^{-1}\mcl{I}$. This proves the claim.
\end{proof}
Upon localization at each $\mcl{M}_i$ in the factorization of $\mcl{I}$, we observe that the
powers are also uniquely determined because of these the two claims. Hence this
Proposition~\ref{prop:UniqueFactorization} follows.
\end{proof}
Now we define the valuation of an ideal in the multiplicative monoid of maximal ideals 
with respect to a maximal ideal.
\begin{defn}[A Total Valuation Map V, Valuation $V_{\mcl{M}}$ at $\mcl{M}$ on Monoid 
$\mathit{M}(\R)$]
\label{defn:Monoid}
Let $\R$ be a commutative ring with unity. Suppose for any maximal 
ideal $\mcl{M}\subs \R, \mcl{M}^i \neq \mcl{M}^{i+1}$ for all $i \geq 0$. Define $\mathit{M}(\R)$ to be the
multiplicative monoid of generated by all maximal ideals in $\R$.

Define two maps \equ{V,V_{\mcl{M}}:\mathit{M}(\R) \lra \mbb{N}\cup
\{0\}} as \equa{V(\mcl{J}&=\us{i=1}{\os{t}{\prod}}\mcl{N}_i^{s_i}
\in \mathit{M}(\R))=\us{i=1}{\os{t}{\sum}} s_i\\
V_{\mcl{M}}(\mcl{J}&=\us{i=1}{\os{t}{\prod}}\mcl{N}_i^{s_i} \in
\mathit{M}(\R))=s_i \text{ if } \mcl{M}=\mcl{N}_i \text{ otherwise }0.}
The above definitions are well defined due to Proposition~\ref{prop:UniqueFactorization}. 
\end{defn}
Here we introduce a definition of factorization associated to an element as a product of maximal ideals 
with respect to a finitely generated monoid. For a fixed element and a monoid this factorization is unique. 
\begin{defn}
\label{defn:Uniquefact}
Let $\R$ be a commutative ring with unity. The ring $\R$ satisfies the
following properties.
\begin{enumerate}
\item For each maximal ideal $\mcl{M}$ we have $\mcl{M}^i \neq
\mcl{M}^{i+1}$ for all $i \geq 0$.
\item $\us{n \geq 0}{\bigcap}\mcl{M}^i=(0)$.
\end{enumerate}
Let $\mathit{M}(\R)$ be the monoid generated by maximal ideals in
$\R$. Let $\mcl{F}$ be a finite set of maximal ideals. Then for any
$0 \neq x \in \R$ we can define map $V_{\mcl{F}}$ with
respect to the monoid $\mcl{F}$. Since $x \neq 0$ for each maximal
ideal $\mcl{M}$ there exists a largest integer $i=i_{\mcl{M}}\geq 0$
such that $x \in \mcl{M}^{i}\bs \mcl{M}^{i+1}$. We define the two maps
\equ{V_{\mcl{F}}:\R^{*} \lra \mathit{M}(\mcl{F}),V_{\mcl{M}}:\R^{*}
\lra \mbb{N}} as $V_{\mcl{F}}(x)=\us{\mcl{M}\in
\mcl{F}}{\prod}\mcl{M}^{i_{\mcl{M}}}$ and
$V_{\mcl{M}}(x)=i_{\mcl{M}}$. Clearly $x \in V_{\mcl{F}}(x)$ and
$V_{\mcl{F}}(x)$ is the unique factorization associated to the element $x$ with
respect to the monoid $\mathit{M}(\mcl{F})$.
\end{defn}

Now we prove the following proposition about non-emptiness of certain sets.
\begin{prop}[Non-emptiness]
\label{prop:NonEmptiness} Let $\R$ be a commutative ring
with identity. For each maximal ideal $\mcl{M}\subs \R$ suppose $\mcl{M}^i \neq \mcl{M}^{i+1}$ for all $i\geq 0$
and $\us{i\geq 0}{\bigcap}\mcl{M}^i=(0)$. Let
$\mcl{F}$ be a finite set of maximal ideals in $\R$. Let
$\mathit{M}(\mcl{F})$ be the finitely generated monoid by $\mcl{F}$. 
Let $\mcl{I}=\mcl{M}_1^{t_1}\mcl{M}_2^{t_2}\ldots\mcl{M}_k^{t_k} \in
\mathit{M}(\mcl{F})$ be a product of maximal ideals. Then the set
\equ{\mcl{I}\bs \bigg(\us{\mcl{J}\in
\mathit{M}(\mcl{F})^{*}}{\bigcup}\mcl{I}\mcl{J}\bigg) \neq \es.}
\end{prop}
\begin{proof}
Using unique factorization of ideals in the monoid $\mathit{M}(\mcl{F})$
as a product of maximal ideals in $\mcl{F}$ we conclude that $\mcl{I}\neq 0$.
Now we can use Proposition~\ref{prop:IdealAvoidance} on ideal
avoidance for the ring $\R$. Since the monoid is finitely generated
by finitely many maximal ideals in $\mcl{F}$, we have
\equ{\mcl{I}\bs \bigg(\us{\mcl{M}\in
\mcl{F}}{\bigcup}\mcl{I}\mcl{M}\bigg)= \mcl{I}\bs
\bigg(\us{\mcl{J}\in
\mathit{M}(\mcl{F})^{*}}{\bigcup}\mcl{I}\mcl{J}\bigg) \neq \es} where
$\mathit{M}(\mcl{F})^{*}$ denote the set $\mathit{M}(\mcl{F})\bs \{\R\}$.
\end{proof}
The following proposition is also similar to the previous proposition and it gives rise to 
multiplicative properties.
\begin{prop}
\label{prop:DetvalElement} Let the notation be as in Proposition~\ref{prop:NonEmptiness}. 
For every ideal $\mcl{I}
\in \mathit{M}(\mcl{F})$, let $a_\mcl{I} \in \mcl{I}\bs
\bigg(\us{\mcl{J}\in
\mathit{M}(\mcl{F})^{*}}{\bigcup}\mcl{I}\mcl{J}\bigg)$. Let
$\mcl{I}_j\in \mathit{M}(\mcl{F}): 1\leq j \leq r$ are pairwise
co-maximal. Then \equ{\us{i=1}{\os{r}{\prod}}a_{\mcl{I}_i} \in
\us{i=1}{\os{r}{\prod}} \mcl{I}_i\bs \bigg(\us{\mcl{J}\in
\mathit{M}(\mcl{F})^{*}}{\bigcup}\mcl{J}\us{i=1}{\os{r}{\prod}}
\mcl{I}_i\bigg)}
\end{prop}
\begin{proof}
First we prove the claim below.
\begin{claim}
\label{claim:unitmodulo}
If $a \in \R$ and $s \nin \mcl{M}$ then \equ{a \in \mcl{M}^i\bs
\mcl{M}^{i+1} \Lra as \in \mcl{M}^i\bs \mcl{M}^{i+1}.}
\end{claim}
\begin{proof}[Proof of Claim]
First we observe that if $s\nin \mcl{M}$ then $s$ is a unit modulo $\mcl{M}^k$ for all $k\geq 0$.
If $a \in \mcl{M}^i$ then $as \in \mcl{M}^i$. If $as \in
\mcl{M}^{i+1}$ then since $s \nin \mcl{M},a \in \mcl{M}^{i+1}$. So
one way implication follows. Now the other way implication also
follows similarly. This proves the claim.
\end{proof}
Since the ideals $\mcl{I}_i: 1\leq i \leq r$ are co-maximal,
for any $1\leq i\neq j \leq r$ we have that $a_{\mcl{I}_i}$ is a unit modulo every power of any maximal ideal containing $\mcl{I}_j$. 
Now the proposition follows because of Claim~\ref{claim:unitmodulo}.
\end{proof}
Proposition~\ref{prop:NonEmptiness} gives rise to the following
definition.
\begin{defn}
Let $\R$ be a commutative ring with unity. Suppose for each maximal
ideal $\mcl{M}$ we have $\mcl{M}^i \neq \mcl{M}^{i+1}$ and $\bigcap
\mcl{M}^i=(0)$. Let $\mcl{F}$ be a finite set of maximal ideals in $\R$.
Let $\mathit{M}(\mcl{F})$ be the finitely generated monoid by the
finite set $\mcl{F}$. Let $\mathit{M}(\mcl{F})^{*}=\mathit{M}(\mcl{F})\bs\{\R\}$. Let $\mcl{I}\in \mathit{M}(\mcl{F})$. Define
the set \equ{\mcl{S}_{\mcl{I}}\eqdef\mcl{I}\bs \bigg(\us{\mcl{J}\in
\mathit{M}(\mcl{F})^{*}}{\bigcup}\mcl{I}\mcl{J}\bigg).} By Proposition~\ref{prop:UniqueFactorization} the ideal $(0) \nin \mathit{M}(\mcl{F})$. By
Proposition~\ref{prop:NonEmptiness} this set
$\mcl{S}_{\mcl{I}}$ is non-empty.
\end{defn}
Now we prove a useful theorem below which produces elements in $S_{\mcl{I}}$
for ideals $\mcl{I}$ in a finitely generated multiplicative monoid which satisfy 
multiplicative properties and co-maximality conditions. The theorem is as stated below.
\begin{theorem}
\label{theorem:Comaximality} Let $\R$ be a commutative ring with
unity. Suppose for each maximal ideal $\mcl{M} \subs \R$ we have
\begin{enumerate}[label=(\alph*)]
\item $\mcl{M}^i \neq \mcl{M}^{i+1}$ for all $i\geq 0$ and
\item $\us{i\geq 0}{\bigcap} \mcl{M}^i=(0)$.
\end{enumerate}
Let $\mcl{F}$ be a finite set of maximal ideals in $\R$.  Let $\mathit{M}(\mcl{F})$ denote the
corresponding finitely generated monoid. Suppose every non-zero element $r \in \R$ is
contained in finitely many maximal ideals. Then there exists a nowhere zero choice
multiplicative monoid map 
$\Gs=\Gs_{\mcl{F}}:\mathit{M}(\mcl{F}) \lra \R$ such that
\begin{enumerate}
\item (Unit Condition): $\Gs(\R)=1$.
\item (Choice Set Condition): $\Gs(\mcl{I}) \in \mcl{S}_{\mcl{I}}$ for all $\mcl{I} \in 
\mathit{M}(\mcl{F})$.
\item (Multiplicativity Condition): If $\mcl{I},\mcl{J} \in \mathit{M}(\mcl{F})$ are 
co-maximal then
$\Gs(\mcl{I}\mcl{J})=\Gs(\mcl{I})\Gs(\mcl{J})$.
\item (Comaximality Condition): For ideals $\mcl{I}_1,\mcl{I}_2,\ldots,\mcl{I}_r \in 
\mathit{M}(\mcl{F})$
\equ{\text{If }\mcl{I}_1+\mcl{I}_2+\ldots+\mcl{I}_r=1\text{ then
}(\Gs(\mcl{I}_1))+(\Gs(\mcl{I}_2))+\ldots+(\Gs(\mcl{I}_r))=1.}
\end{enumerate}
\end{theorem}
\begin{proof}
We prove this theorem as follows.
\begin{claim}
If $\mcl{I},\mcl{J}\in \mathit{M}(\mcl{F})$ are co-maximal then we
have $(\mcl{S}_{\mcl{I}})+(\mcl{S}_{\mcl{J}})=1$ i.e. the ideals of
the sets are co-maximal and may not be the sets themselves.
\end{claim}

\begin{proof}[Proof of Claim]
Let $\mcl{M}$ be a maximal ideal containing the set
$\mcl{S}_{\mcl{I}}$ then $\mcl{M}$ occurs in the unique
factorization of $\mcl{I} \in \mathit{M}(\mcl{F})$. Suppose not, then 
\equ{\mcl{I}=\big(\mcl{I}\cap \mcl{M}\big)\us{\mcl{N} \in
\mcl{F}}{\bigcup}\mcl{I}\mcl{N}=\mcl{I}\mcl{M}\us{\mcl{N} \in
\mcl{F}}{\bigcup}\mcl{I}\mcl{N}}
contradicting Proposition~\ref{prop:IdealAvoidance} for the non-zero ideal $\mcl{I}$.
Since there are no common maximal
ideals occurring in the unique factorization of $\mcl{I},\mcl{J}$ the
claim follows.
\end{proof}
Now we prove on more claim.
\begin{claim}
Let $j\in \mbb{N}, \mcl{M}\in \mcl{F}$. Let $a\in \mcl{S}_{\mcl{M}^j}$. Then the maximal ideals which contain $a$ are $\mcl{M}$ and some finitely maximal ideals $\mcl{N}\subs \mcl{R}$ such that 
$\mcl{N} \nin \mcl{F}\bs\{\mcl{M}\}$.
\end{claim}
\begin{proof}[Proof of Claim]
Let $\mcl{M}\in \mcl{F}$ be any maximal ideal. Fix a $j\in \mbb{N}$. 
Let $a\in \mcl{S}_{\mcl{M}^j}$. Then clearly we have $a\in \big(\mcl{M}^j\bs \mcl{M}^{j+1}\big)\subs \mcl{M}$.  Also the element $a$ naturally avoids any maximal ideal $\mcl{N}\in \mcl{F}\bs\{\mcl{M}\}$,
that is, for any maximal ideal $\mcl{N}\in \mcl{F},\mcl{N}\neq \mcl{M}$ we have $a\nin \mcl{N}$.
This is because \equ{\mcl{S}_{\mcl{M}^j} \subs \mcl{M}^j \bs \mcl{M}^j\mcl{N}} using the definition of $\mcl{S}_{\mcl{M}^j}$. If $a\in \mcl{N}$ then $a\in \mcl{M}^j\cap \mcl{N}=\mcl{M}^j\mcl{N}$
which is a contradiction.  So the maximal ideals which contain $a$ are $\mcl{M}$ and possibly some other maximal ideals $\ti{\mcl{M}}$ in $\mcl{R}$ such that $\ti{\mcl{M}}\nin \mcl{F}\bs\{\mcl{M}\}$. 
Hence the claim follows. 
\end{proof}

Continuing with the proof, we define $\Gs(\R)=1$. Let
$\mcl{F}=\{\mcl{M}_1,\mcl{M}_2,\ldots,\mcl{M}_k\}$. Since every
non-zero element is contained in finitely many maximal ideals we
find the elements $\Gs(\mcl{M}_i^{t_i})\in S_{\mcl{M}_i^{t_i}}, t_i\in \mbb{N}, 1\leq i\leq k$
inductively as follows.

First we choose any $\Gs(\mcl{M}_1) \in S_{\mcl{M}_1}$. Now this
element is contained in finitely many maximal ideals. Choose
$\Gs(\mcl{M}_2) \in S_{\mcl{M}_2}$ avoiding these finitely many
maximal ideals. Continue this process till we find 
$\#(\mcl{F})=k\operatorname{-}$elements $\Gs(\mcl{M}_i) \in \mcl{S}_{\mcl{M}_i}$ inductively for $1 \leq i \leq k$. 
At each $i^{th}$ step the element $\Gs(\mcl{M}_i)$ also avoids every maximal ideal containing any of the previous 
$i-1$ elements $\Gs(\mcl{M}_j),1\leq j\leq i-1$. These $k\operatorname{-}$elements $\Gs(\mcl{M}_i),1\leq i\leq k$ will be pairwise co-maximal. 
Here we do not want to just extend multiplicatively by raising powers in general. 

Now we find $\Gs(\mcl{M}_1^2) \in S_{\mcl{M}_1^2}$ which is co-maximal to all the previously found
elements $\Gs(\mcl{M}_2),\ldots,\Gs(\mcl{M}_k)$ corresponding to other maximal ideals $\mcl{M}_2,\ldots,\mcl{M}_k$ using Proposition~\ref{prop:IdealAvoidance}
on ideal avoidance. Also $\Gs(\mcl{M}_1^2)$ must avoid every maximal ideal other than $\mcl{M}_1$ containing
$\Gs(\mcl{M}_1)$. So continuing this way we have defined $\Gs$ for
all powers of maximal ideals in $\mcl{F}$. Now extend $\Gs$
multiplicatively to the entire monoid for co-maximal ideals. We use
Proposition~\ref{prop:DetvalElement} to conclude $\Gs(\mcl{I}) \in
\mcl{S}_{\mcl{I}}$.

Now if $\mcl{I}_1+\mcl{I}_2+\ldots+\mcl{I}_r=1$. Let $\mcl{M}$ be
any maximal ideal. If $\mcl{M}$ contains all the elements
$\Gs(\mcl{I}_1),\Gs(\mcl{I}_2),\ldots,\Gs(\mcl{I}_r)$ then $\mcl{M}$
contains $\Gs(\mcl{M}_i^{l_i})$ and $\Gs(\mcl{M}_j^{l_j})$ for two
distinct maximal ideals $\mcl{M}_i\neq \mcl{M}_j$ in $\mcl{F}$ which is a contradiction. So
co-maximality condition follows.

Now the fact that $\Gs(\mcl{I}) \in S_{\mcl{I}}$ implies that $\Gs$
is nowhere zero. Now Theorem~\ref{theorem:Comaximality} follows.
\end{proof}
\begin{obs}
\label{obs:DistinctMax} In Theorem~\ref{theorem:Comaximality} while
defining the map $\Gs_{\mcl{F}}$ it satisfies the following properties
automatically. Let $\{\mcl{M}_i, 1\leq i\leq l\}\subs \mcl{F}$.
\begin{enumerate}[label=(\alph*)]
\item If $\mcl{A}=\mcl{M}_1^{t_1}\mcl{M}_2^{t_2}\ldots\mcl{M}_l^{t_l},\mcl{B}=\mcl{M}_1^{s_1}\mcl{M}_2^{s_2}\ldots\mcl{M}_l^{s_l}$
with $\mcl{M}_1,\ldots,\mcl{M}_l\in \mcl{F}$ then 
\equ{\Gs_{\mcl{F}}(\mcl{A})=\Gs_{\mcl{F}}(\mcl{M}_1^{t_1})\ldots \Gs_{\mcl{F}}(\mcl{M}_l^{t_l}),\Gs_{\mcl{F}}(\mcl{B})=\Gs_{\mcl{F}}(\mcl{M}_1^{s_1})\ldots
\Gs_{\mcl{F}}(\mcl{M}_l^{s_l}).} 
\item We have for each $1 \leq i \neq j \leq l$, the set of maximal ideals containing $\Gs_{\mcl{F}}(\mcl{M}_i^{t})$ is disjoint from the set of maximal ideals containing
$\Gs_{\mcl{F}}(\mcl{M}_j^{s})$ for any $t,s\geq 0$. 
\item If $t \neq s$ for any $1 \leq i \leq l$, then the set of maximal ideals containing $\Gs_{\mcl{F}}(\mcl{M}_i^{t})$ is disjoint from the set of maximal ideals containing
$\Gs_{\mcl{F}}(\mcl{M}_i^{s})$ other than $\mcl{M}_i$.
\end{enumerate} 
\end{obs}
\begin{example}
\begin{itemize}
\item Let $\R=\Z$. Here $\Gs$ can be defined for the entire monoid
$\mathit{M}(\R)$. The map $\Gs:\mathit{M}(\R) \lra \R$ given by
$\Gs((p_1^{t_1}p_2^{t_2}\ldots p_k^{t_k}))=p_1^{t_1}p_2^{t_2}\ldots
p_k^{t_k},$ where $p_i:1 \leq i \leq k$ are $k\operatorname{-}$distinct primes.
\item Let $\R$ be a Dedekind domain with finitely many maximal ideals which is not
a field. It is a principal ideal domain. Any element in $\gp_i \in \mcl{P}_i \bs
\bigg(\us{j\neq i}{\bigcup}\mcl{P}_j\cup \mcl{P}_i^2\bigg)$ is a
generator as its ideal factorization in $\R$ is given by
$(\gp_i)=\mcl{P}_i$. Here the monoid $\mathit{M}(\R)$ is finitely
generated. Then define
$\Gs(\us{i=1}{\os{k}{\prod}}\mcl{P}_i^{t_i})=\us{i=1}{\os{k}{\prod}}\gp_i^{t_i}$.
\item A Dedekind domain $\R$ is a principal ideal domain if and only if
for every maximal ideal $\mcl{M}$, the set \equ{\mcl{M}\bs
\bigg(\big(\us{\mcl{N}\in \max(Spec(R)),\mcl{N} \neq
\mcl{M}}{\bigcup}\mcl{N}\big)\bigcup \mcl{M}^2\bigg) \neq \es.} Then
we could define the map $\Gs$ similar to the ring of integers
explicitly.
\end{itemize}
\end{example}


\subsection{\bf{On unital sets modulo an ideal in Dedekind type domain}}
We begin this sections with a remark. 
\begin{remark}
Let $\R$ be a commutative ring with unity. Let $k>0$ be a positive
integer. Let $(a_1,a_2,\ldots,a_{k+1})$ be a unital set in $\R$.
Suppose $a_1x_1+a_2x_2+\ldots+a_kx_k+a_{k+1}x_{k+1}=1$ and
$\{x_1,x_2,\ldots,x_k\}$ is also a unital set. i.e.
$b_1x_1+b_2x_2+\ldots+b_kx_k=1$ then we have
\equ{(a_1+a_{k+1}x_{k+1}b_1)x_1+(a_2+a_{k+1}x_{k+1}b_2)x_2+\ldots+
(a_k+a_{k+1}x_{k+1}b_k)x_k=1}
i.e. there exist $t_1,t_2,\ldots,t_k\in (a_{k+1})$ such that the
set $\{a_1+t_1,a_2+t_2,\ldots,a_k+t_k\}$ is unital in $\R$.
\end{remark}
Now prove the following two important useful
Propositions~[\ref{prop:CMHNDDIMONE},~\ref{prop:UnitalDD}] on unital sets modulo an ideal in a Dedekind type domain.
\begin{prop}
\label{prop:CMHNDDIMONE} Let $\R$ be a Dedekind type domain.
Let $\mcl{I} \sbnq \R$ be an ideal. Let $x,y \in \R$.
Suppose $(x)+(y)+\mcl{I}=\R$. Then there exist $a,b \in \R$ such that
$ax+by\equiv 1\mod \mcl{I}$ and $(a)+(b)=\R$.
\end{prop}
\begin{proof}
Suppose $a_1x+b_1y+i=1$ for some $i \in \mcl{I}$. 
Let $\mcl{M}_1,\mcl{M}_2,\ldots,\mcl{M}_r$ be the common maximals ideals in $\R$ containing $a_1,b_1$. 
Let $\mcl{N}_1,\mcl{N}_2,\ldots,\mcl{N}_s$ be the remaining maximal ideals in $\R$ that contain $a_1$. Then $i\nin \us{k=1}{\os{r}{\bigcup}} \mcl{M}_k$.
Let $\mcl{J}=(i)\us{k=1}{\os{s}{\prod}}\mcl{N}_k$. So $\mcl{J}\nsubseteq \mcl{M}_k, 1\leq k\leq r$. Then by Proposition~\ref{prop:IdealAvoidance} choose 
$j\in \mcl{J}\bs \us{k=1}{\os{r}{\bigcup}}\mcl{J}\mcl{M}_k \neq \es$. Then $j\in (i)\subs \mcl{I}$. 
Since $b_1\in \mcl{M}_k,j\nin \mcl{M}_k \Ra b_1+j\nin \mcl{M}_k, 1\leq k\leq r$. Also $b_1\nin \mcl{N}_k,j\in \mcl{N}_k\Ra b_1+j\nin \mcl{N}_k, 1\leq k\leq s$.
Let $a=a_1,b=b_1+j$. This implies $(a)+(b)=\R$. Also \equ{ax+by=a_1x+b_1y+jy\equiv a_1x+b_1y \equiv 1\mod \mcl{I}.} Now the proposition follows.
\end{proof}
\begin{prop}
\label{prop:UnitalDD} Let $\R$ be a Dedekind type domain.
Let $\mcl{I}\sbnq \R$ be an ideal. Let $r>1$ be a positive integer.
Suppose $(a_1,a_2,\ldots,a_r)$ is a unital set modulo $\mcl{I}$.
Then there exist $t_1,t_2,\ldots,t_r\in \mcl{I}$ such that the set
$\{a_1+t_1,\ldots,a_r+t_r\}$ is unital in $\R$.
\end{prop}
\begin{proof}
Let $a_1x_1+\ldots+a_rx_r+i=1$ for $i \in \mcl{I}$. If $i=0$ then
there is nothing to prove. So assume $i \neq 0$.

Suppose two of the $x_j's$ are non-zero. Say $x_1 \neq 0,x_2 \neq
0$. Let $e=a_2x_2+\ldots+a_rx_r+i$. Then $a_1x_1+e=1$.  By Lemma~\ref{lemma:FundLemma} there exists $t\in \R$ such that
$(x_1-te)+(x_2)=\R$. We also have $(x_1-te)a_1+(1+ta_1)e=1$. So we
get
\equ{a_1(x_1-te)+a_2x_2(1+ta_1)+a_3x_3(1+ta_1)+\ldots+a_rx_r(1+ta_1)+i(1+ta_1)=1}
Now we have both
$(x_1-te)+(x_2)=\R,(x_1-te)+(1+ta_1)=\R$ so $(x_1-te)+(x_2(1+ta_1))=\R$. There exist $s_1,s_2 \in \R$
such that \equ{(x_1-te)s_1+x_2(1+ta_1)s_2=1 \Ra
(x_1-te)s_1i(1+ta_1)+x_2(1+ta_1)s_2i(1+ta_1)=i(1+ta_1)} Hence we get
\equ{(a_1+s_1(1+ta_1)i)(x_1-te)+(a_2+s_2(1+ta_1)i)x_2(1+ta_1)+a_3x_3(1+ta_1)+\ldots+
a_rx_r(1+ta_1)=1}
So choosing $t_1=is_1(1+ta_1),t_2=is_2(1+ta_2)\in
\mcl{I},t_3=t_4=\ldots=0$ we get $\{a_i+t_i:1 \leq i \leq r\}$ is a
unital set.

Suppose all but one of the $x_i$ is zero. Say $x_1 \neq 0$ and
$x_2,x_3,\ldots,x_r=0$. Then $a_1x_1+i=1$ and suppose $a_j=0$ for
some $j\geq 2$. Then choose $t_j=i,t_l=0$ for $l \neq j$ and we have
the set $\{a_1,a_2,\ldots,a_{j-1},a_j+t_j,a_{j+1},\ldots,a_r\}$ is
unital.

Now if $x_1\neq 0, x_2=x_3=\ldots=x_r=0,a_2,a_3,\ldots,a_r \neq 0$
and $r \geq 3$ then we could choose $x_2=a_3,x_3=-a_2$ and we have
at least two of the $x_j's$ non-zero which is considered before.

For the case $r=2$ let $(a_1)+(a_2)+\mcl{I}=\R$. Now using 
Proposition~\ref{prop:CMHNDDIMONE} we have that there exist $x_1,x_2$
such that $(x_1)+(x_2)=\R$ and $a_1x_1+a_2x_2+i=1$ for some $i \in
\mcl{I}$. So if $x_1y_1+x_2y_2=1$ then $x_1y_1i+x_2y_2i=i$. So we
get $\{a_1+y_1i,a_2+y_2i\}$ is a unital set.
 This completes the proof of Proposition~\ref{prop:UnitalDD}.
\end{proof}


\section{\bf{On the surjectivity of the Chinese remainder reduction map}}
\label{sec:CRTSURJMOSTGENCASE}
In order to prove surjectivity of the map in Theorem~\ref{theorem:CRTSURJMOSTGENCASE}
we first observe that the image is invariant under a suitable action of the two 
groups for $k\in \mbb{N}$
\begin{enumerate}
\item $SL_{k+1}=\{A\in M_{(k+1)\times (k+1)}(\R)\mid \Det(A)=1\}$ and
\item $\ti{S}L_{k+1}=\{A\in M_{(k+1)\times (k+1)}(\R)\mid \Det(A)=\pm 1\}$.
\end{enumerate}


\subsection{$SL_{k+1}\operatorname{-}$\bf{Invariance of the image of the Chinese
remainder reduction map}}
Here we define the action of $SL_{k+1}$ on $\mbb{PF}^{k}_{\mcl{I}}$.
\begin{defn}[$SL_{k+1}\operatorname{-}$action]
\label{defn:SLKAction}
Let $\R$ be a commutative ring with unity. Let \equ{\mcl{I}
\in\mcl{INTRAD}(\R)^{*}.} There is a well defined left action of
$SL_{k+1}(R)$ as follows. Let $g \in SL_{k+1}(\R)$. Define
\equ{L_g=r_{g^{-1}}:\mbb{PF}^{k}_{\mcl{I}}\lra\mbb{PF}^{k}_{\mcl{I}}}
given by $L_g([a_0:a_1:a_2\ldots:a_k])=g\bullet
([a_0:a_1:a_2\ldots:a_k])=r_{g^{-1}}([a_0:a_1:a_2\ldots:a_k])=[b_0:b_1:b_2:\ldots:b_k],$
where \equ{(b_0,b_1,b_2,\ldots,b_k)=(a_0,a_1,a_2,\ldots,a_k)g^{-1}.}
This action can be extended to a product of such projective spaces in a similar manner.
\end{defn}
\begin{lemma}[$SL_{k+1}\operatorname{-}$Invariance of the image]
\label{lemma:SLkRInvariance} Let $\R$ be a commutative ring with
unity. Let $\mcl{I}_i \in \mcl{INTRAD}(\R)^{*}: 1\leq i \leq n$ be
finitely many pairwise co-maximal ideals in $\R$. Let
\equ{\mcl{I}=\us{i=1}{\os{n}{\prod}}\mcl{I}_i.} The image of the
Chinese remainder reduction map is a union of $SL_{k+1}\operatorname{-}$orbits.
\end{lemma}
\begin{proof}

If \equ{\gs:\mbb{PF}^k_{\mcl{I}} \lra
\us{i=1}{\os{n}{\prod}}\mbb{PF}^k_{\mcl{I}_i}} then the Chinese
remainder reduction map $\gs$ is always $SL_{k+1}\operatorname{-}$invariant in the
sense that for any $g\in SL_{k+1}(\R)$ we have
\equ{g\bullet\gs([a_0:a_1:a_2\ldots:a_k])=\gs(g\bullet[a_0:a_1:a_2\ldots:a_k]).}
Hence this result follows.
\end{proof}
\begin{note}
Let $\ti{S}L_{k+1}(\R)=\{A \in M_{k+1}(\R)\mid \Det(A)=\pm 1\}$. We can
similarly conclude like in Lemma~\ref{lemma:SLkRInvariance} that the
image of the Chinese remainder reduction map is $\ti{S}L_{k+1}(\R)\operatorname{-}$
invariant and it is a union of $\ti{S}L_{k+1}(\R)\operatorname{-}$orbits.
\end{note}


\subsection{Surjectivity of the Chinese remainder reduction map} 
Here in this section we prove the first
main Theorem~\ref{theorem:CRTSURJMOSTGENCASE} of this article.
\begin{proof}
The theorem holds for $k=1$ and any $l>0$ as the proof is immediate. Now we prove by induction on $k$. Let
\equ{([a_{10}:a_{11}:\ldots:a_{1l}],\ldots,[a_{k0}:a_{k1}:\ldots:a_{kl}])
\in \mbb{PF}^{l}_{\mcl{Q}_1} \times \mbb{PF}^{l}_{\mcl{Q}_2}
\times \ldots \times \mbb{PF}^{l}_{\mcl{Q}_k}} By induction we
have an element $[b_0:b_1:b_2:\ldots:b_l] \in
\mbb{PF}^{l}_{\mcl{Q}_2\mcl{Q}_3\ldots\mcl{Q}_k}$ representing the
last $k-1$ elements. Now consider the matrix \equ{A=\begin{pmatrix}
{\mcl{Q}_1 \lra} & a_{10}  & a_{11}  & \cdots & a_{1,l-1} & a_{1l}\\
{\mcl{Q}_2\ldots \mcl{Q}_k \lra} & b_0& b_1& \cdots & b_{l-1} & b_l
\end{pmatrix}}
Now one of the elements in the first row is not in $\mcl{M}_1$. By
finding inverse of this element modulo $\mcl{Q}_1$ and hence by a
suitable application of $\ti{S}L_{l+1}(\R)$ matrix the matrix $A$ can
be transformed to the following matrix $B$, where we replace the
unique non-zero entry in the first row by $1$.
\equ{B=\begin{pmatrix}
{\mcl{Q}_1 \lra} & 1  & 0  & \cdots & 0 & 0\\
{\mcl{Q}_2\ldots \mcl{Q}_k \lra} & c_0& c_1& \cdots & c_{l-1} & c_l
\end{pmatrix}}
If $c_0$ is a unit $\mod \mcl{Q}_2\ldots \mcl{Q}_k$ then let $vc_0\equiv 1 \mod \us{i=2}{\os{k}{\prod}}\mcl{Q}_i$.
Now the matrix $B$ represents the same elements as the matrix 
\equ{B'=\begin{pmatrix}
{\mcl{Q}_1 \lra} & 1  & 0  & \cdots & 0 & 0\\
{\mcl{Q}_2\ldots \mcl{Q}_k \lra} & 1& vc_1& \cdots & vc_{l-1} & vc_l
\end{pmatrix}}
So this reduces to usual Chinese Remainder Theorem.

Otherwise, if $c_0$ is not unit $\mod \mcl{Q}_2\ldots \mcl{Q}_k$ then suppose without loss of generality \equ{c_0 \in \mcl{M}_2\mcl{M}_3\ldots\mcl{M}_r\bs
\mcl{M}_{r+1}\cup\mcl{M}_{r+2}\cup\ldots\cup\mcl{M}_k.} Let
$\us{i=0}{\os{l}{\sum}}c_ix_i=1$. Now consider any element $a \in
\mcl{M}_{r+1}\ldots\mcl{M}_k\bs (\mcl{M}_2\cup\ldots\cup\mcl{M}_r) \neq
\es$. Then the matrix \equ{C=\begin{pmatrix}
{\mcl{Q}_1 \lra} & 1  & 0  & \cdots & 0 & 0\\
{\mcl{Q}_2\ldots \mcl{Q}_k \lra} &
c_0+\us{i=1}{\os{l}{\sum}}ac_ix_i=a+c_0(1-ax_0) & c_1& \cdots &
c_{l-1} & c_l
\end{pmatrix}}
is obtained from $B$ by $\ti{S}L_{l+1}(\R)$ matrix. Now the element
\equ{a+c_0(1-ax_0)\nin \mcl{M}_2 \cup \ldots \cup \mcl{M}_k.} Let $u
\in R$ be such that $u(a+c_0(1-ax_0))\equiv 1 \mod
\us{i=2}{\os{k}{\prod}}\mcl{Q}_i$. Then the matrix $C$ represents
the same elements as the matrix $D$. \equ{D=\begin{pmatrix}
{\mcl{Q}_1 \lra} & 1  & 0  & \cdots & 0 & 0\\
{\mcl{Q}_2\ldots \mcl{Q}_k \lra} & 1& uc_1& \cdots & uc_{l-1} & uc_l
\end{pmatrix}}

The elements in the matrix $C$ is in the image of Chinese remainder reduction map
by the usual Chinese Remainder Theorem.

Hence the induction step is completed and Theorem~\ref{theorem:CRTSURJMOSTGENCASE} follows.
\end{proof}
\begin{example}[Construction of a counter example for surjectivity in one dimension]
\label{example:CounterExample}
Let $\R=\mbb{K}[x,y],$ where $\mbb{K}$ is a field. Consider the prime
ideals $\mcl{P}_1=(x-1)$, $\mcl{P}_2=(y-1)$. We note that these are
not finite intersection of ideals whose radicals are maximal ideals.
However here we observe that
$\mcl{P}_1\mcl{P}_2=\mcl{P}_1\cap \mcl{P}_2$ by unique factorization
domain property. Now consider the Chinese remainder reduction map
\equ{\mbb{PF}^1_{\mcl{P}_1\mcl{P}_2} \lra \mbb{PF}^1_{\mcl{P}_1}
\times \mbb{PF}^1_{\mcl{P}_2}} This map is not surjective.
Consider the element $([1:0],[0:1]) \in \mbb{PF}^1_{\mcl{P}_1}
\times \mbb{PF}^1_{\mcl{P}_2}$. If $a,b \in \R$ represent this
element via congruence conditions then we get

\equa{a \equiv 1\mod (x-1),a \equiv 0 \mod (y-1)\\
b \equiv 0\mod (x-1),b \equiv 1 \mod (y-1)}

So we get $a=(y-1)t$ and $a-1=-(x-1)u$ implying $(y-1)t+(x-1)u=1$ which yields a contradiction. 
So via congruences we cannot obtain a representing element pair $(a,b)$. 
Now let $a,b \in \R$ generate the unit ideal such that 
$[a:b]=[1:0] \in\mbb{PF}^1_{\mcl{P}_1}$ and $[a:b]=[0:1] \in \mbb{PF}^1_{\mcl{P}_2}$
then $(x-1)\mid b,(y-1) \mid a$. So we have the ideal $(a,b)\subs(x-1,y-1)$ 
which is impossible. This proves that the Chinese remainder reduction map is not
surjective.
\end{example}


\section{\bf{Surjectivity of the map $SL_k(\R) \lra SL_k(\frac{\R}{\mcl{I}})$ and the unital 
set condition with respect to an ideal}}
\label{sec:SurjModIdeal}
In this section we consider the reduction map
\equ{SL_k(\R) \lra SL_k(\frac{\R}{\mcl{I}})} and prove the second main Theorem~\ref{theorem:SurjModIdeal} of
this article.
\begin{proof}
For $k=1$ the proof is immediate. So assume $k>1$. Clearly all
elementary matrices $E_{ij}(r),r \in \R, i \neq j$ are in the image.
Now consider a diagonal matrix
$diag(d_{11}=d_1,d_{22}=d_2,\ldots,d_{kk}=d_k)$ such that
\equ{d_{1}d_{2}\ldots d_{k} \equiv 1 \mod \mcl{I}.} Let
$n=d_{1}d_{2}\ldots d_{k}-1 \in \mcl{I}$.

Define a matrix

\equ{E=\matk e}

with $e_{k1}=nz,e_{12}=e_{23}=e_{34}=\ldots=e_{(k-1)k}=n$. Let
\equ{e_{ii}=d_{i}+ \ga^i_1n + \ga^i_2n^2+ \ldots
\ga^i_{k-1}n^{k-1}\in \R[\ga^i_j:1 \leq i \leq k,1 \leq j \leq
(k-1)]} be a polynomial representing a symbolic respective $n\operatorname{-}$adic
expansion modulo $(n^k)$. Choose the rest of the entries in the
matrix $E$ to be zero. Now this matrix has determinant given by
\equ{e_{11}e_{22}\ldots e_{kk} - (-1)^{k}n^{k}z.} The sum of ideals
$(e_{11}e_{22}\ldots e_{kk})+(n^k)=(1)$ in the polynomial ring
$\R[\ga^i_j:1 \leq i \leq k,1 \leq j \leq (k-1)]$ because
$(e_{11}e_{22}\ldots e_{kk})+(n)=(d_{1}d_{2}\ldots d_{k})+(n)=(1)$
and using radical of ideals. i.e.
\equa{rad(A+rad(B))&=rad(rad(A)+B)\\
&=rad(rad(A)+rad(B))=rad(A+B) \text{ for ideals }A,B.} So there exist $w,\ga\in \R[\ga^i_j:1 \leq i \leq k,1 \leq j
\leq (k-1)]$ such that \equ{\ga e_{11}e_{22}\ldots e_{kk} + wn^k=1.}
If we choose for the symbols $\ga^i_j$ elements of $R$ such that
\equ{e_{11}e_{22}\ldots e_{kk} \equiv 1\mod n^k} then we get $\ga
\equiv 1\ mod\ n^k$. So we can solve for $z$ so that the determinant
\equ{e_{11}e_{22}\ldots e_{kk} - (-1)^{k}n^{k}z=1.}

To solve first consider $k=2$. If $d_1d_2=1+t_1n+t_2n^2 + \ldots +
(n^k)$ be its symbolic $n\operatorname{-}$adic expansion then we should have
$\ga^1_1d_2+\ga^2_1d_1+t_1 \equiv 0 \mod\ n$. Such an equation is
solvable say for $\ga^1_1$ or for $\ga^2_1$ as $d_1,d_2$ are units
mod $n^r$ for all $r$. To obtain a value $t_1$ we know that
$d_1d_2-1=n\ti{t}_1$ for some $\ti{t}_1\in \R$. So choose
$t_1=\ti{t}_1$ and there are no remaining $t_i$ as $k=2$ here in
this case.

For a general $k$. Let the symbolic $n\operatorname{-}$adic expansions be given by
\equa{d_1d_2\ldots d_k=1+t_1n+t_2n^2+\ldots+t_kn^{k-1}+(n^k),\\
d_2d_3\ldots d_k= s_0+s_1n+s_2n^2+\ldots+s_{k-1}n^{k-1}+(n^k)\\
e_{11}=d_1+\ga_1n+\ga_2n^2+\ldots+\ga_{k-1}n^{k-1}+ (n^k).}

Fix a section $sec:\frac{\R}{(n)} \lra \R$. Recursively pick
representative values in the image of $sec$ in $\R$ for $t_i$ for
$i=1,\ldots,(k-1)$, and $s_i$ for $i=0,\ldots,(k-1)$. Let
$e_{ii}=d_i$ for all $i \geq 2$ then \equ{e_{11}e_{22}\ldots
e_{kk}=d_1d_2\ldots d_k+\ga_1nd_2d_3\ldots d_k+\ga_2n^2d_2d_3\ldots
d_k+\ldots+ (n^k).}

So we should have $s_0\ga_1+t_1 \equiv 0 \mod\ n$. So solve for
$\ga_1$ as $s_0$ is a unit $\mod\ n$. Now solve for $\ga_2$ because
$s_0\ga_2+\ldots \equiv 0 \mod\ n$ recursively by carrying the
addendum of the previous term $s_0\ga_1+t_1$ which are higher
powers of $n$ and so on for the rest of the $\ga_i's$. The $\ga_i$
gets multiplied by $s_0$ which is a unit mod $n$. So solving for
$\ga_i$ is possible.

We have proved that the diagonal determinant one matrices in
$SL_k(\frac{\R}{\mcl{I}})$ are in the image of the reduction map
$\gs:SL_k(\R) \lra SL_k(\frac{\R}{\mcl{I}})$ by choosing
$n=d_1d_2\ldots d_k-1 \in \mcl{I}$ for each
$diag(d_1,d_2,\ldots,d_k) \in SL_k(\frac{\R}{\mcl{I}})$.

Now we prove the following claim. We note here that $k>1$.
\begin{claim}
All matrices in $SL_k(\frac{\R}{\mcl{I}})$ can be reduced to identity
by elementary determinant one matrices and matrices of the form
$diag(1,\ldots,u,u^{-1},\ldots,1),$ where $u \in
\mcl{U}(\frac{\R}{\mcl{I}})$ a unit if $\mcl{I}$ satisfies the unital
set condition.
\end{claim}
\begin{proof}[Proof of Claim]
To prove this we observe that we can reduce any element to identity
using elementary matrices and matrices of the form
\equ{diag(1,\ldots,u,u^{-1},\ldots,1),} where $u \in
\mcl{U}(\frac{\R}{\mcl{I}})$ a unit. This reduction can be done
because if $(a_1,a_2,\ldots,a_k)$ is a row then there exists an
element $i \in \mcl{I}$ such that $\{a_1,a_2,\ldots,a_k,i\}$ is
unital. Since $\mcl{I}$ satisfies the $USC$ there exists
$j \in (a_2,\ldots,a_k,i)$ such that $a_1+j$ is a unit modulo
$\mcl{I}$. Now the element $i$ can be ignored so that we can bring a
unit $\mod{I}$ in a row by applying only elementary determinant one
matrices as column operations. This proves the claim for
$SL_k(\frac{\R}{\mcl{I}})$.
\end{proof}

Continuing with the proof of main Theorem~\ref{theorem:SurjModIdeal} , we observe that
all matrices are in the image i.e. the reduction map 
$\gs:SL_k(\R) \lra SL_k(\frac{\R}{\mcl{I}})$ is onto. This finishes the proof of
Theorem~\ref{theorem:SurjModIdeal}.
\end{proof}

\begin{cor}
\label{cor:SurjModIdeal} Let $\R$ be a commutative ring with unity.
Let $\mcl{I} \sbnq \R$ be an ideal contained in finitely many maximal
ideals. Then the reduction map \equ{SL_k(\R) \lra
SL_k(\frac{\R}{\mcl{I}})} is onto for any $k\in \mbb{N}$.
\end{cor}
\begin{proof}[Proof of Corollary]
For $k=1$ the proof is immediate. For $k>1$ this corollary
follows from the fact that any ideal $\mcl{I}$ which is contained in
finitely many maximal ideals satisfies the $USC$ using Proposition~\ref{prop:Unital}.
\end{proof}


\subsection{\bf{An important consequence of unital lemma}}
\label{sec:ImportantConsequenceUnitalLemma}
In this section we prove the following proposition.
\begin{prop}[A proposition on elementary row vector of dimension more than one]
\label{prop:ElementaryRowVector} Let $\mcl{\R}$ be a commutative
ring with unity. Let $\mcl{I}$ be an ideal which is contained only
in finitely many maximal ideals. Let $k>1$ be a positive integer.
Let $\{a_1,a_2,\ldots,a_k\} \subs \R$ be a unital set i.e.
$\us{i=1}{\os{k}{\sum}}(a_i)=\R$. Then there exists a matrix $g$ in
$SL_{k}(\R)$ such that \equ{(a_1,a_2,\ldots,a_k)g\equiv
(1,0,\ldots,0) \mod \mcl{I}} For $k=1$ the existence of such a
matrix $g$ need not hold.
\end{prop}
We begin with a lemma which is stated as follows.
\begin{lemma}
\label{lemma:wedge} Let $\R$ be a ring. Let $k>1$ be a positive
integer. Let $(a_1,a_2,\ldots,a_k) \in \R^k$ be a vector such that
$a_i$ is a unit for some $1 \leq i \leq k$. Then there exist
$k\operatorname{-}$vectors $\{v_1,v_2,\ldots,v_k\} \subs \R^{k-1}$ such that
\equ{v_1\wedge v_2\wedge \ldots \wedge \widehat{v_i}\wedge \ldots
\wedge v_k=a_i}
\end{lemma}
\begin{proof}
First consider a unital vector $(a_1,a_2,\ldots,a_k)$ with $a_1$ a
unit without loss of generality. Let
\equa{
v_1&=(a_2,-a_1^{-1}a_3,+a_1^{-1}a_4,\ldots,(-1)^ia_1^{-1}a_i,\ldots,(-1)^ka_1^{-1}a_k)^t\\
&=a_2e^{k-1}_1+\us{i=2}{\os{k-1}{\sum}}(-1)^{i+1}a_1^{-1}a_{i+1}e^{k-1}_i,
v_2=a_1e^{k-1}_1,v_3=e_2^{k-1},\ldots,v_k=e_{k-1}^{k-1}.
} 
Then we immediately observe that for $1 \leq i \leq k$,\equ{v_1\wedge
v_2\wedge \ldots \wedge \widehat{v_i}\wedge \ldots \wedge v_k=a_i}
Similarly if any other component $a_i$ is a unit. Hence the lemma
follows.
\end{proof}
Now we prove Proposition~\ref{prop:ElementaryRowVector} of this section.
\begin{proof}
For $\mcl{I}=\R$ the proof is easy. So assume $\mcl{I}\sbnq \R$.
We note that if $k=1$ and $a_1$ is a unit in $\R$ but $a_1 \not\equiv
1 \mod \mcl{I}$ then $a_1g \equiv 1 \mod \mcl{I}$ does not imply
that $g \in SL_1(\R)$ unless $1+\mcl{I}$ is the set of all units in $\R$.

Now assume $k>1$. Let $(b_1,b_2,\ldots,b_k) \in \R^k$ such that
$\us{i=1}{\os{k}{\sum}}(-1)^{i-1}a_ib_i=1$. Now the vector
$(b_1,b_2,\ldots,b_k)$ is unital. So from Lemma~\ref{lemma:Unital} 
there exist $t_2,t_3,\ldots,t_k\in \R$ such
that the element $c_1=b_1+t_2b_2+\ldots+t_kb_k$ is a unit modulo
$\mcl{I}$.

Now consider the vector $(c_1,b_2,\ldots,b_k)$ which has a unit
$\mod \mcl{I}$. Hence using Lemma~\ref{lemma:wedge} there exist
$k\operatorname{-}$vectors $\{v_1,v_2,\ldots,v_k\} \subs \R^{k-1}$ such that
$v_2\wedge v_3 \wedge \ldots \wedge v_k \in c_1+\mcl{I}$ and
\equ{v_1\wedge v_2\wedge \ldots \wedge \widehat{v_i}\wedge \ldots
\wedge v_k \in b_i +\mcl{I} \text{ if } i>1} Now choose
\equ{w_1=v_1,w_2=v_2-t_2v_1,w_3=v_3+t_3v_1,\ldots,
w_k=v_k+(-1)^{k-1}t_kv_1} Then we have for $i \geq 2$ \equ{w_1\wedge
w_2\wedge \ldots \wedge \widehat{w_i}\wedge \ldots \wedge w_k \in
b_i +\mcl{I}} and $w_2\wedge w_3 \wedge \ldots \wedge w_k \in
b_1+\mcl{I}$. So the following matrix has unit determinant modulo
$\mcl{I}$. i.e. treating each $w_i$ is a column $(k-1)\operatorname{-}$ vector we
have \equ{\Det\mattwofour
{a_1}{a_2}{\ldots}{a_k}{w_1}{w_2}{\ldots}{w_k} \equiv 1\ \mod
\mcl{I}}

So using Theorem~\ref{theorem:SurjModIdeal} there exists a matrix $B
\in SL_k(\R)$ such that we have \equ{B \equiv \mattwofour
{a_1}{a_2}{\ldots}{a_k}{w_1}{w_2}{\ldots}{w_k} \mod \mcl{I}.} We
observe that \equ{(1,0,\ldots,0)B\equiv (a_1,a_2,\ldots,a_k)\mod
\mcl{I}.} So we consider $g=B^{-1}$ and this proposition follows.
\end{proof}

\begin{remark}
If $\mcl{R}$ is a commutative ring with unity and $\mcl{I}\sbnq \mcl{R}$ is an
ideal which is contained in finitely many maximal ideals then for $k>1$ 
Proposition~\ref{prop:ElementaryRowVector} proves that the set
\equ{\{(\ol{a_1},\ol{a_2},\ldots,\ol{a_k})\in
\bigg(\frac{\mcl{R}}{\mcl{I}}\bigg)^k\mid \{a_1,a_2,\ldots,a_k\}\subs \mcl{R}
\text{ is a unital set}\}}
is a transitive orbit under the action of  $SL_k(\mcl{R})$.	
\end{remark}


\section{\bf{Representation of elements of projective spaces associated to ideals}}
\label{sec:RepHigherDimElement}
In this section we prove the following Theorem~\ref{theorem:RepHigherDimElement}
which is stated below. 
\begin{theorem}
\label{theorem:RepHigherDimElement} Let $\R$ be a commutative ring
with unity. Suppose $\R$ is a Dedekind type domain (refer to
Definition~\ref{defn:Good}). Let $\mathit{M}(\R)$ be the monoid
generated by maximal ideals in $\R$. Let $\mcl{I}\in \mathit{M}(\R)$
be a product of maximal ideals. Let $k \geq 2$ be a positive
integer. Let $\mcl{F}$ be any finite set of maximal ideals
containing $V(\mcl{I})$. Let $\Gs=\Gs_{\mcl{F}}$ be the nowhere zero choice monoid
multiplicative map for the monoid $\mathit{M}(\mcl{F})$ from
Theorem~\ref{theorem:Comaximality}. Then the description of the $k\operatorname{-}$
dimensional projective space is given by
\equa{\mbb{PF}^k_{\mcl{I}}&=\{[\Gs(\mcl{J}_0)v_0:\Gs(\mcl{J}_1)v_1:\ldots:\Gs(\mcl{J}_k)v_k]
\bigg\vert v_i\in \bigg(\R\bs \us{\mcl{M}\in
\mcl{F}}{\bigcup}\mcl{M}\bigg),\mcl{I}\subseteq \mcl{J}_i \in \mathit{M}(\mcl{F}),1\leq i\leq k\\
&\us{i=0}{\os{k}{\sum}}\mcl{J}_i=\R,\us{i=0}{\os{k}{\sum}}(\Gs(\mcl{J}_i)v_i)=\R\}.}
\end{theorem}
\begin{obs}
In the above description of the projective space $\mbb{PF}^k_{\mcl{I}}$, 
either of the conditions $\us{i=0}{\os{k}{\sum}}\mcl{J}_i=\R,\us{i=0}{\os{k}{\sum}}(\Gs(\mcl{J}_i)v_i)=\R$ 
implies that $\us{i=0}{\os{k}{\sum}}(\Gs(\mcl{J}_i))=\R$.  
\end{obs} 


We first consider one dimensional projective spaces before we proceed to higher 
dimensions.
\subsection{Representation of elements in one dimensional projective space associated 
to ideals}
~\\
We begin with a couple of Lemmas~\ref{lemma:AReplemma},\ref{lemma:unit} before 
proving Theorem~\ref{theorem:RepOneDimension} of representing elements of one dimensional
projective spaces. For any ring $\R$ let $\mcl{U}(\R)$ denote the set of units in $\R$.
\begin{lemma}[A representation lemma]
\label{lemma:AReplemma} Let $\R$ be a commutative ring with unity. Let $\mcl{M}$
be a maximal ideal. Let $k$ be a positive integer. Suppose
$dim_{\frac{\R}{\mcl{M}}}(\frac{\mcl{M}^t}{\mcl{M}^{t+1}})=1$ for $0
\leq t \leq (k-1)$. Let $p_t \in \mcl{M}^t\bs \mcl{M}^{t+1},0\leq t \leq (k-1)$. 
Then the projective space
\equa{\mbb{PF}^1_{\mcl{M}^k}&=\bigg\{[1:p_{t}u]\mid
\bar{u}\in\mcl{U}(\frac{\R}{\mcl{M}^{k-t}}), 0\leq t \leq (k-1)\bigg\}\\
&\bigcup \bigg\{[p_tu:1]\mid\bar{u}\in\mcl{U}(\frac{\R}{\mcl{M}^{k-t}}), 0\leq t \leq (k-1)\bigg\}\\
&\bigcup \bigg\{[1:0],[0:1]\bigg\}}
\end{lemma}
\begin{proof}
Clearly if $[a:b] \in \mbb{PF}^1_{\mcl{M}^k}$ then either $a \nin
\mcl{M}$ or $b \nin \mcl{M}$. So without loss of generality we can
assume either $a=1$ or $b=1$. So assume $a=1$. Then
$[1:b_1]=[1:b_2]$ if and only if $b_1-b_2 \in \mcl{M}^k$. Moreover
for each $i=1,2$, either $b_i \in (\mcl{M}^t\bs\mcl{M}^{t+1})$ for some $0 \leq t<k$, 
or $b_i \in \mcl{M}^k$. Also for any $0
\leq t < k$ \equ{b_1 \in \mcl{M}^{t}\bs \mcl{M}^{t+1}
\Lra b_2 \in \mcl{M}^{t}\bs \mcl{M}^{t+1}} and for $t=k$
\equ{b_1 \in \mcl{M}^{k} \Lra b_2 \in \mcl{M}^{k}.}
Now let $b \in \mcl{M}^t\bs \mcl{M}^{t+1}$ for $0\leq t<k$ and let $b= p_tu +
\mcl{M}^{t+1}$. Here $u$ actually can be varied in a coset of $\mcl{M}$ because, if \equ{p_tu + \mcl{M}^{t+1}= p_tu^{'} +
\mcl{M}^{t+1}} then we have $u-u^{'} \in \mcl{M}$ since $\{\ol{p_t}\}$ is a basis for $\frac{\mcl{M}^t}{\mcl{M}^{t+1}}$.
Thus if $t+1=k, [1:b]=[1:p_tu]$. Now let $t+1<k$. 
Here first we observe that \equ{b- p_tu=\us{l}{\sum} x_ly_l\text{ with
}x_l\in \mcl{M}^t,y_l\in \mcl{M}.} Now again expressing each $x_l$
in terms of the basis $\{p_t\}$ modulo $\mcl{M}^{t+1}$ and repeating
this process and increasing the powers of $\mcl{M}$ for $y's$ from $x's$ and fixing the power $t$ of $\mcl{M}$ for $x's$ we can 
reach $\mcl{M}^k$. So we can actually assume that \equ{b=
p_{t}v+\mcl{M}^k} for possibly some other $v \nin \mcl{M}$.
So we have represented and described all elements of $\mbb{PF}^1_{\mcl{M}^k}$. 
For $[1:b]=[1:p_tv]$, if $k\geq t+1$ then we can actually vary $v$ in its coset of $\mcl{M}^{k-t}$
without changing the projective element $[1:b]$.

This proves Lemma~\ref{lemma:AReplemma}.
\end{proof}

\begin{lemma}[An observation between the addition and multiplication in the ring]
\label{lemma:unit} Let $\R$ be a commutative ring with unity. Let
$\mcl{I}=\mcl{M}^k,$ where $\mcl{M}\subs \R$ be a maximal ideal.
Suppose for $0\leq t \leq k-1, dim_{\frac{\R}{\mcl{M}}}(\frac{\mcl{M}^t}{\mcl{M}^{t+1}})=1$.
Let $p_i,\ti{p}_i \in \mcl{M}^i\bs \mcl{M}^{i+1}$ for $0 \leq i
\leq k$. Then for any $u \in \R\bs \mcl{M}$ there exists $v \in \R\bs
\mcl{M}$ such that $[1:p_iu]=[1:\ti{p}_iv] \in
\mbb{PF}^1_{\mcl{I}}$. i.e. $p_iu-\ti{p}_iv \in \mcl{M}^{k}$.
Moreover $u$ and $v$ can be varied in their respective cosets $\mod
\mcl{M}^{k-i}$ without changing the element in the projective space
$\mbb{PF}^1_{\mcl{M}^k}$.
\end{lemma}
\begin{proof}
This lemma follows from Lemma~\ref{lemma:AReplemma}.
\end{proof}
Now we state the following theorem of representing elements for one dimensional 
projective spaces.
\begin{theorem}
\label{theorem:RepOneDimension} Let $\R$ be a commutative ring with
unity. Suppose $R$ is a Dedekind type domain (refer to
Definition~\ref{defn:Good}). Let $\mathit{M}(\R)$ be the monoid
generated by maximal ideals in $\R$. Let $r\in \mbb{N}$ and 
\equ{\mcl{I}=\mcl{M}_1^{k_1}\mcl{M}_2^{k_2}\ldots\mcl{M}^{k_r}_r\in\mathit{M}(\R)}
be an ideal. Let $\mcl{F}$ be any finite set of maximal ideals
containing $V(\mcl{I})$. Let $\Gs=\Gs_{\mcl{F}}$ be the nowhere zero choice monoid
multiplicative map for the monoid $\mathit{M}(\mcl{F})$ from
Theorem~\ref{theorem:Comaximality}.
Then the projective space
\equa{\mbb{PF}^1_{\mcl{I}}&=\{[\Gs(\mcl{I}_1):\Gs(\mcl{I}_2)u]\bigg\vert
u \in \R\bs\bigg( \us{\mcl{M} \in \mcl{F}}\bigcup \mcl{M}\bigg), \text{ where for }i=1,2 \ \ \mcl{I}\subseteq \mcl{I}_i \in
\mathit{M}(\mcl{F})\\ &\mcl{I}_1+\mcl{I}_2=\R,(\Gs(\mcl{I}_1))+(\Gs(\mcl{I}_2)u)=\R\}.} 
\end{theorem}
\begin{proof}
Consider an element $e=(e_1,e_2,\ldots,e_r) \in
\us{i=1}{\os{r}{\prod}}\mbb{PF}^1_{\mcl{M}_i^{k_i}}$. Let $A \sqcup
B$ be a partition of the set $\{1,2,\ldots,r\}$ such that if $i \in
A$ then $e_i=[1:\Gs(\mcl{M}_i^{j_i})u_i]$ for some $u_i\nin
\mcl{M}_i$ and if $i \in B$ then $e_i=[\Gs(\mcl{M}_i^{j_i})v_i:1]$
for some $v_i\nin \mcl{M}_i$. Here $0 \leq j_i \leq k_i$. This
representation holds for $e$ using Lemma~\ref{lemma:AReplemma}. Using the Chinese remainder reduction
isomorphism in Theorem~\ref{theorem:CRTSURJMOSTGENCASE} there exists
an element $[a:b]\in \mbb{PF}^1_{\mcl{I}}$ such that $[a:b] = e_i \in \mbb{PF}^1_{\mcl{M}_i^{k_i}}$. 
Actually by usual Chinese Remainder Theorem and by an application of Proposition~\ref{prop:UnitalDD}, 
we can actually find $[a:b]\in \mbb{PF}^1_{\mcl{I}}$ such that \equa{&a\equiv 1 \mod \mcl{M}_i^{k_i}, b \equiv \Gs(\mcl{M}_i^{j_i})u_i\mod \mcl{M}_i^{k_i} \text{ for }i\in A \text{ and }\\
&b \equiv 1 \mod \mcl{M}_i^{k_i},a \equiv \Gs(\mcl{M}_i^{j_i})v_i\mod \mcl{M}_i^{k_i} \text{ for }i\in B.}
Let $\mcl{I}_1 = \us{i\in
B}{\prod}\mcl{M}_i^{j_i},\mcl{I}_2 = \us{i\in
A}{\prod}\mcl{M}_i^{j_i}$. We observe that
$\mcl{I}_1,\mcl{I}_2$ are co-maximal as $A,B$ are disjoint. Now we
factor $\Gs(\mcl{I}_1),\Gs(\mcl{I}_2)$ from $a,b$ respectively using
congruences especially using Lemma~\ref{lemma:unit}. Let $i \in A$.
Now $a \equiv 1\mod \mcl{M}_i^{k_i},b \equiv
\Gs(\mcl{M}_i^{j_i})u_i\mod \mcl{M}_i^{k_i}$. Let $t
\Gs(\mcl{I}_1)\equiv 1 \mod \mcl{M}_i^{k_i}$. We observe both
$b,\Gs(\mcl{I}_2)t \in \mcl{M}_i^{j_i}\bs \mcl{M}_i^{j_i+1}$ unless
$j_i=k_i$ in which case both $b,\Gs(\mcl{I}_2)t \in
\mcl{M}_i^{k_i}$. Now we use Lemma~\ref{lemma:unit} to conclude that
that there exists $x_i \in \R\bs \mcl{M}_i$ such that
$b-\Gs(\mcl{I}_2)tx_i\in \mcl{M}_i^{k_i}$. This proves that
\equ{[a:b]=[1:\Gs(\mcl{I}_2)tx_i]=[\Gs(\mcl{I}_1):\Gs(\mcl{I}_2)x_i]
\in \mbb{PF}^1_{\mcl{M}_i^{k_i}}.} We can do similarly if $i \in B$.
So we have factored $\Gs(\mcl{I}_1),\Gs(\mcl{I}_2)$ from $a,b$ for
all $1 \leq i \leq r$ respectively obtaining suitable elements $x_i
\in \R\bs \mcl{M}_i$ for $1 \leq i \leq r$. So we get that
\equ{[a:b]=(e_1,e_2,\ldots,e_r)=([\Gs(\mcl{I}_1):\Gs(\mcl{I}_2)x_1],
[\Gs(\mcl{I}_1):\Gs(\mcl{I}_2)x_2],\ldots,[\Gs(\mcl{I}_1):\Gs(\mcl{I}_2)x_r])}
Now we obtain the element $u \in \R\bs\bigg( \us{\mcl{M} \in
\mcl{F}}\bigcup \mcl{M}\bigg)$ as follows. We solve three sets of
congruences simultaneously.
\begin{itemize}
\item The first set of congruences is \equ{u \equiv x_i\mod
\mcl{M}^k_i}
\item The second set of congruences is as follows.  For $\mcl{M} \in
\mcl{F}\bs \{\mcl{M}_1,\mcl{M}_2,\ldots,\mcl{M}_r\}$, \equ{u \equiv
1\mod \mcl{M}}
\item The third set of congruences is as follows. Since any element $r\in
R$ is in finitely many maximal ideals, let $\mcl{G}$ be the finite
set of maximal ideals which contain $\Gs(\mcl{I}_1),\Gs(\mcl{I}_2)$.
Then we solve for $\mcl{M}\in \mcl{G}\bs \mcl{F}$ \equ{u \equiv
1\mod \mcl{M}}
\end{itemize}
So by solving these congruences for $u$ we have the following.
\begin{enumerate}[label=(\alph*)]
\item $u \in \R\bs\bigg(\us{\mcl{M} \in \mcl{F}}\bigcup \mcl{M}\bigg)$.
\item There is no common maximal ideal containing any two of the elements $u,\Gs(\mcl{I}_1),\Gs(\mcl{I}_2)$.
\end{enumerate} 
So $[\Gs(\mcl{I}_1):\Gs(\mcl{I}_2)u] \in \mbb{PF}^1_{\mcl{I}}$ is not only a well defined element but also a required element. 
This proves Theorem~\ref{theorem:RepOneDimension}.
\end{proof}
\begin{obs}
\label{obs:coset}
In the above description of $\mbb{PF}^1_{\mcl{I}}$ we have 
\equ{[\Gs(\mcl{I}_1):\Gs(\mcl{I}_2)u]=[\Gs(\mcl{I}_1):\Gs(\mcl{I}_2)\ti{u}]}
where $\ti{u}\in u+\mcl{I}_3$ with $\mcl{I}_1\mcl{I}_2\mcl{I}_3=\mcl{I}$ provided 
$(\Gs(\mcl{I}_1))+(\Gs(\mcl{I}_2)\ti{u})=\R$. 
\end{obs}


\subsection{Representation of elements in higher dimensional projective space associated to 
ideals}
~\\
Now we prove the main Theorem~\ref{theorem:RepHigherDimElement} of 
this section about representing elements of projective spaces associated to ideals of
any dimension.
\begin{proof}
Let $\mcl{I}=\mcl{M}_1^{t_1}\mcl{M}_2^{t_2}\ldots\mcl{M}_l^{t_l}\in
\mathit{M}(\R)$. Let $[x_0:x_1:\ldots:x_k]\in \mbb{PF}^k_{\mcl{I}}$.
Assume each $x_i$ is non-zero by replacing the element by a non-zero
element of $\mcl{I}$. This also does not alter the condition
$\us{i=0}{\os{k}{\sum}}(x_i)=\R$. We define the ideal $\mcl{J}_i$ as
follows. Let $\mcl{G}=\{\mcl{M}_1,\ldots,\mcl{M}_l\}=V(\mcl{I})$.
Consider the unique factorizations associated to $x_i$ with respect to the
monoid $\mathit{M}(\mcl{G})$ (refer to Definition~\ref{defn:Uniquefact}). Define the ideal \equ{\text{for }0
\leq i\leq k,\mcl{J}_i=\us{j=1}{\os{l}{\prod}}\mcl{M}_j^{min(t_j,V_{\mcl{M}_j}(x_i))}\Ra
\mcl{J}_i \sups V_{\mcl{G}}(x_i) \sups \{x_i\},\mcl{J}_i \sups
\mcl{I}.} So $\us{i=0}{\os{k}{\sum}}(\mcl{J}_i)=\R$. Hence we also
have $\us{i=0}{\os{k}{\sum}}(\Gs(\mcl{J}_i))=\R$ for
$\Gs:\mathit{M}(\mcl{F})\lra \R,$ where $\mcl{F} \sups \mcl{G}$. Now
we factor $\Gs(\mcl{J}_i)$ from $x_i$ for $0 \leq i \leq k$ using
congruences. First for a fixed $1 \leq j\leq l$, using 
Lemma~\ref{lemma:unit} we conclude that there exist $v_{ij}\in
\R\bs\mcl{M}_j$ such that $x_i-\Gs(\mcl{J}_i)v_{ij}\in
\mcl{M}_j^{t_j}$ if $t_j>V_{\mcl{M}_j}(x_i)$. Note if $V_{\mcl{M}_j}(x_i) \geq t_j$ then we
could choose $v_{ij}=1$. By Chinese Remainder Theorem for a fixed
$i,0\leq i\leq k$ we lift $v_{ij}$ to an element $v_i\in \R\bs \us{\mcl{M}\in
\mcl{F}}{\bigcup}\mcl{M}$ by solving congruences. \equ{v_i\equiv
v_{ij}\mod \mcl{M}_j^{t_j}, 1\leq j\leq l} We may need to solve some additional
finitely many congruences for $0\leq i\leq k$ successively of the type \equ{v_i \equiv 1\mod \mcl{N}}
to avoid a maximal ideal $\mcl{N}$ and also to ensure the condition
that \equ{\us{i=0}{\os{k}{\sum}}(\Gs(\mcl{J}_i)v_i)=\R} which can be
done as every non-zero element is contained in finitely many maximal
ideals. We note that we can choose $v_i, 0\leq i\leq k$ such that the set of maximal ideals that contain $v_i$
is disjoint from that of $v_j$ for $0\leq i\neq j \leq k$.
Hence Theorem~\ref{theorem:RepHigherDimElement} follows.
\end{proof}


\section{\bf{Surjectivity of the map $SL_k(\R) \lra \us{i=1}{\os{k}{\prod}}\mbb{PF}^{k-1}_{\mcl{I}_i}$}}
Here in this section we prove the third main result, surjectivity Theorem~\ref{theorem:SurjSLkHighDimension} for all $k\geq 2$.
First we consider an explicit computation in a slightly general scenario over an arbitrary commutative ring with unity.


\subsection{\bf{Surjectivity for a pair of maximal ideals in arbitrary commutative 
ring with unity}} 
\label{sec:SEMIACRU}
~\\
Here we describe explicitly a collection of $2 \times 2$
determinant one matrices which map onto the product of spaces
$\mbb{PF}^1_{\mcl{N}} \times \mbb{PF}^1_{\mcl{M}}$ for two maximal
ideals $\mcl{N},\mcl{M}$ in the ring $\R$ and thereby proving the following proposition.
\begin{prop}
\label{prop:SEMIACRU}
Let $\R$ be a commutative ring with unity. Let $\mcl{M},\mcl{N}$ be two distinct maximal ideals in $\R$. Then the maps 
$\gs_i:SL_2(\R) \lra \mbb{PF}^1_{\mcl{N}} \times \mbb{PF}^1_{\mcl{M}},i=1,2$ (refer to Theorem~\ref{theorem:SurjSLkHighDimension}) 
are surjective.
\end{prop}
\begin{proof}
Fix any two sections $s_{\mcl{N}}:\frac{\R}{\mcl{N}} \lra \R$ and
$s_{\mcl{M}}:\frac{\R}{\mcl{M}} \lra \R$ of the quotient maps
$\gt_{\mcl{M}}:\R \lra \frac{\R}{\mcl{M}},\gt_{\mcl{N}}:\R \lra
\frac{\R}{\mcl{N}}$.
Consider the following set of matrices
\equ{\mcl{C}_1=\bigg\{\mattwo {s}{(st-1)}{1}{t},s \in
image(s_{\mcl{N}}),t \in image(s_{\mcl{M}})\bigg\}}

This set of matrices maps into the subset \equ{\mbb{PF}^1_{\mcl{N}}
\times \bigg( \{[1:t] \in \mbb{PF}^1_{\mcl{M}}\mid t \in
image(s_{\mcl{M}})\}\bigg) \subs \mbb{PF}^1_{\mcl{N}} \times
\mbb{PF}^1_{\mcl{M}}} injectively giving rise to distinct elements.

\equ{\mcl{C}_1 \hookrightarrow \mbb{PF}^1_{\mcl{N}} \times \bigg(
\{[1:t] \in \mbb{PF}^1_{\mcl{M}}\mid t \in
image(s_{\mcl{M}})\}\bigg)}

There is one more element for each $t \in image(s_{\mcl{M}})$ with
$[1:t]$ as the image corresponding to the second row. It is given as
follows. Since $\mcl{M},\mcl{N}$ are co-maximal there exist elements
$p \in \mcl{M}$, $q \in \mcl{N}$ such that the ideals $(p),(q)$ are
co-maximal i.e. $(p)+(q)=1$. Consider elements $r,q \in \R$ such that
$rq-kp=1$ as $(p)+(q)=1$ and for such $p,q,r,k$, we have that the
ideals $(p(1+qr)),(q(1+pk))$ are co-maximal. So consider elements
$l,m$ such that $lp(1+qr)-mq(1+pk)=1-t$ for any given $t \in \R$. Now
consider $2 \times 2$ matrices of determinant $1$.
\equ{\mcl{C}_2=\bigg\{\mattwo {(1+rq)}{(t+mq)}{(1+kp)}{(t+lp)}, t
\in image(s_{\mcl{M}})\bigg\}}

Now the collection $\mcl{C}_1\cup \mcl{C}_2$ maps injectively into
the set $\mbb{PF}^1_{\mcl{N}} \times \{[1:t] \in
\mbb{PF}^1_{\mcl{M}}\mid t \in image(s_{\mcl{M}})\}$. We shall soon
observe that this collection actually maps onto this set
bijectively. i.e

\equ{\big(\mcl{C}_1 \cup \mcl{C}_2\big) \cong  \mbb{PF}^1_{\mcl{N}}
\times \bigg( \{[1:t] \in \mbb{PF}^1_{\mcl{M}}\mid t \in
image(s_{\mcl{M}})\}\bigg)}

Now consider the set \equ{\mcl{C}_3=\bigg\{\mattwo
{(1+sp)}{s}{p}{1},s \in image(s_{\mcl{N}})\bigg\}}

This set maps injectively into the set $\mbb{PF}^1_{\mcl{N}} \times
\{[0:1]\}$

\equ{\mcl{C}_3 \hookrightarrow \mbb{PF}^1_{\mcl{N}} \times
\{[0:1]\}}

We will soon see that the set $\mcl{C}_3$ misses just one element in
the set $\mbb{PF}^1_{\mcl{N}} \times \{[0:1]\}$.

Now we describe that one more matrix of determinant one which maps
onto the missing element $([p:1],[0:1]) \in \mbb{PF}^1_{\mcl{N}}
\times \mbb{PF}^1_{\mcl{M}}$. Consider elements $k,r \in \R$ such
that $rq-kp=1$ as $(p)+(q)=1$. For such integers $k,p,r,q$ we have
that the ideals $(p(1+rq)),(q(1+kp))$ are co-maximal. So consider
elements $m,l \in \R$ such that $lq(1+kp)-mp(1+rq)=1-p-kp^2$. Then
consider $2 \times 2$ matrix of determinant $1$ given by
\equ{\mattwo {(lq+p)}{(1+rq)}{mp}{(1+kp)}}

Now we observe that we have a total collection of two by two
matrices of determinant one mapping injectively into
$\mbb{PF}^1_{\mcl{N}} \times \mbb{PF}^1_{\mcl{M}}$.

We immediately see that for a fixed $t \in image(s_{\mcl{M}})$
\equ{\{[s:st-1]\mid s \in image(s_{\mcl{N}})\}= \{[1:w]\mid w \in
image(s_{\mcl{N}}),[1:w] \neq [1:t]\}\cup \{[0:1]\}.}

We also observe that \equ{\{[1+sp:s]\mid s \in image(s_{\mcl{N}})\}=
\{[1:w]\mid w \in image(s_{\mcl{N}}),[1:w] \neq [p:1]\}\cup
\{[0:1]\}.}

Hence the mapping $\gs_1$ is onto and similarly the map $\gs_2$ is
also onto. So the intermediate claims of surjectivity of $\mcl{C}_1
\cup \mcl{C}_2$ and the set $\mcl{C}_3$ just missing one element are
justified.
\end{proof}


\subsection{\bf{Surjectivity of the map $SL_2(\R) \lra \mbb{PF}^1_{\mcl{I}} \times 
\mbb{PF}^1_{\mcl{J}}$}}
\label{sec:SurjSL2Case}
In this section we prove the surjectivity Theorem~\ref{theorem:SurjSLkHighDimension}
for $k=2$ which is stated below. It is enough to prove that $\gs_1$ is surjective.
\begin{theorem}
\label{theorem:SurjSL2Case} Let $\R$ be a commutative ring with
unity. Suppose $\R$ is a Dedekind type domain (refer to
Definition~\ref{defn:Good}). Let $\mathit{M}(\R)$ be the monoid
generated by maximal ideals in $\R$. Let $\mcl{I},\mcl{J}\in
\mathit{M}(\R)$ be two co-maximal proper ideals. Then the map \equ{\gs_1:SL_2(\R)
\lra \mbb{PF}^1_{\mcl{I}} \times \mbb{PF}^1_{\mcl{J}}} given by
\equ{\mattwo abcd \lra ([a:b],[c:d])} is surjective.
\end{theorem}
\begin{proof}
Consider the two co-maximal ideals
\equ{\mcl{I}=\us{i=1}{\os{r}{\prod}}\mcl{M}_i^{k_i},
\mcl{J}=\us{i=1}{\os{s}{\prod}}\mcl{N}_i^{l_i}\in
\mathit{M}(\R).} Let
\equ{\mcl{F}=\{\mcl{M}_1,\mcl{M}_2,\ldots,\mcl{M}_r,\mcl{N}_1,\mcl{N}_2,\ldots,\mcl{N}_s\}.}
Let $\Gs=\Gs_{\mcl{F}}$ be the choice monoid multiplicative map for the monoid
$\mathit{M}(\mcl{F})$ from Theorem~\ref{theorem:Comaximality}. Using
Theorem~\ref{theorem:RepOneDimension} consider an
element
\equ{([\Gs(\mcl{I}_1):\Gs(\mcl{I}_2)u],[\Gs(\mcl{J}_1):\Gs(\mcl{J}_2)v])
\in \mbb{PF}^1_{\mcl{I}_1} \times \mbb{PF}^1_{\mcl{I}_2},} where
$\mcl{I}_1,\mcl{I}_2,\mcl{J}_1,\mcl{J}_2\in \mathit{M}(\R)$ with
$\mcl{I} \subs \mcl{I}_1,\mcl{I}_2$ which are co-maximal and
$\mcl{J} \subs \mcl{J}_1,\mcl{J}_2$ which are co-maximal, where $u,v
\in \R\bs\bigg(\us{\mcl{M} \in \mcl{F}}{\bigcup} \mcl{M}\bigg)$. Let
$\mcl{I}_3,\mcl{J}_3 \in \mathit{M}(\R)$ be the unique ideals such
that $\mcl{I}_1\mcl{I}_2\mcl{I}_3=\mcl{I},\mcl{J}_1\mcl{J}_2\mcl{J}_3=\mcl{J}$.
Let \equ{x \in \R\bs \bigg(\us{i=1}{\os{r}{\bigcup}}\mcl{M}_i\bigg),y
\in \R\bs \bigg(\us{i=1}{\os{s}{\bigcup}}\mcl{N}_i\bigg),i_3 \in
\mcl{I}_3,j_3 \in \mcl{J}_3} and consider the following matrix
\equ{\mattwo
{\Gs(\mcl{I}_1)x}{\Gs(\mcl{I}_2)(xu+i_3)}{\Gs(\mcl{J}_1)y}{\Gs(\mcl{J}_2)(yv+j_3)}.}
Now we solve for $x,y,i_3,j_3$ such that the above matrix has
determinant one (also refer to Observation~\ref{obs:coset}). 
For this purpose let $\ga,\gb \in \R,I_3 \in
\mcl{I}_3,J_3 \in \mcl{J}_3,i_3=I_3\gb \Gs(\mcl{I}_1),j_3=J_3\ga
\Gs(\mcl{J}_1)$ and consider the equation
\begin{equation}
\label{eq:DetEquation}
\begin{aligned}
\Gs(\mcl{I}_1)\Gs(\mcl{J}_2)x(J_3\ga
\Gs(\mcl{J}_1))-\Gs(\mcl{I}_2)\Gs(\mcl{J}_1)y(I_3\gb
\Gs(\mcl{I}_1))=1+(\Gs(\mcl{I}_2)\Gs(\mcl{J}_1)u-\Gs(\mcl{I}_1)\Gs(\mcl{J}_2)v)xy
\end{aligned}
\end{equation}
Consider the co-maximal ideals
\equ{\mcl{K}_1=(\Gs(\mcl{I}_1)),\mcl{K}_2=(\Gs(\mcl{J}_1))} Now we
solve the following congruences for $A \in \R$ given by
\equ{1+\Gs(\mcl{I}_2)\Gs(\mcl{J}_1)uA \in
\mcl{K}_1,1-\Gs(\mcl{I}_1)\Gs(\mcl{J}_2)vA \in \mcl{K}_2} Such
solutions exist because the two pairs of ideals
\begin{itemize}
\item $(\Gs(\mcl{I}_2)\Gs(\mcl{J}_1)u)$ and $\mcl{K}_1$,
\item $(\Gs(\mcl{I}_1)\Gs(\mcl{J}_2)v)$ and $\mcl{K}_2$
\end{itemize}
are also co-maximal. If $A_0$ is
one common solution then the set of common solutions is given by
\equ{A_0+\mcl{K}_1\mcl{K}_2=\{A_0+a \mid a \in \mcl{K}_1\mcl{K}_2\}}
because $\frac{\R}{\mcl{K}_1\mcl{K}_2} \cong \frac{\R}{\mcl{K}_1}
\oplus \frac{\R}{\mcl{K}_2}$. Moreover we have the sum of the ideals
$(A_0)+\mcl{K}_1\mcl{K}_2=\R$. So let $(A_0)+(B_0)=\R$ for some $B_0
\in \mcl{K}_1\mcl{K}_2$. Here in Proposition~\ref{prop:FundLemmaSchemes} we choose the set
\equ{E=V(\mcl{I})\cup V(\mcl{J})\cup V(\Gs(\mcl{I}_2)) \cup
V(\Gs(\mcl{J}_2)),\text{ a finite set}.} Here choice multiplicative monoid map $\Gs$
never takes a zero value. Now we note that
$\Gs(\mcl{I}_2)\Gs(\mcl{J}_2)=\Gs(\mcl{I}_2\mcl{J}_2) \neq 0$ by
multiplicativity.
So using Proposition~\ref{prop:FundLemmaSchemes} there
exists an element of the form $C_0=A_0+nB_0$ for some $n \in \R$ such
that \equ{(C_0)+\mcl{I}\mcl{J}(\Gs(\mcl{I}_2)\Gs(\mcl{J}_2))=\R.}

Now choose $x=1,y=C_0$ in their respective sets such that their
associated principal ideals are co-maximal and also
co-maximal to each ideal $\mcl{I},\mcl{J}$. We observe that
\equ{1+(\Gs(\mcl{I}_2)\Gs(\mcl{J}_1)u-\Gs(\mcl{I}_1)\Gs(\mcl{J}_2)v)xy
\in \mcl{K}_1 \cap \mcl{K}_2=\mcl{K}_1\mcl{K}_2 =
(\Gs(\mcl{I}_1)\Gs(\mcl{J}_1))=(\Gs(\mcl{I}_1\mcl{J}_1)).}

Now let
$1+(\Gs(\mcl{I}_2)\Gs(\mcl{J}_1)u-\Gs(\mcl{I}_1)\Gs(\mcl{J}_2)v)xy
=\Gs(\mcl{I}_1\mcl{J}_1)t$. We solve for $I_3\gb,J_3\ga$ in the following
equation which is obtained from equation~\ref{eq:DetEquation}.
\equ{\Gs(\mcl{J}_2)J_3\ga-\Gs(\mcl{I}_2)C_0I_3\gb=\Gs(\mcl{J}_2)xJ_3\ga-\Gs(\mcl{I}_2)yI_3\gb =t}

Now consider the two ideals
$\Gs(\mcl{J}_2)x\mcl{J}_3=\Gs(\mcl{J}_2)\mcl{J}_3,\Gs(\mcl{I}_2)C_0\mcl{I}_3$. They are
co-maximal. This is a consequence of the following. The pairs of ideals $\{(\Gs(\mcl{I}_2)),(\Gs(\mcl{J}_2))\},
\{(\Gs(\mcl{I}_2)),\mcl{J}_3\},\linebreak \{\mcl{I}_3,(\Gs(\mcl{J}_2))\},\{\mcl{I}_3,\mcl{J}_3\}$ 
are co-maximal and the pairs of ideals $\{(C_0),(\Gs(\mcl{J}_2))\},\{(C_0),\mcl{J}\}$ are
also co-maximal.  Hence the pair of ideals $\{(C_0),\mcl{J}_3\}$
is also co-maximal. So solving for $I_3\gb \in \mcl{I}_3,J_3\ga \in
\mcl{J}_3$ is possible in the above equation. This proves Theorem~\ref{theorem:SurjSL2Case}.
\end{proof}


\subsection{\bf{Surjectivity of the map $SL_k(\R) \lra 
\us{i=1}{\os{k}{\prod}}\mbb{PF}^{k-1}_{\mcl{I}_i}$ for all $k\geq 2$}}
\label{sec:SurjSLkHighDimension}
Here in this section now we prove the third main
Theorem~\ref{theorem:SurjSLkHighDimension} of this article.
\begin{proof}
Let $\gs_1$ be as defined in Theorem~\ref{theorem:SurjSLkHighDimension}.
We have that the image of $\gs_1$ is $SL_k(\R)$ invariant using Lemma~\ref{lemma:SLkRInvariance}
under the well defined action of $SL_k(\R)$ on $\us{i=1}{\os{k}{\prod}}\mbb{PF}^{k-1}_{\mcl{I}_i}$ 
given in Definition~\ref{defn:SLKAction}. Now we prove the following claim.
\begin{claim}
\label{claim:SurjProj} The image of $\gs_1$ equals
$\us{i=1}{\os{k}{\prod}}\mbb{PF}^{k-1}_{\mcl{I}_i}$.
\end{claim}
\begin{proof}[Proof of Claim]
Let
$([a_{11}:a_{12}:\ldots:a_{1k}],[a_{21}:a_{22}:\ldots:a_{2k}],\ldots,[a_{k1}:a_{k2}:\ldots:
a_{kk}])
\in \us{i=1}{\os{k}{\prod}}\mbb{PF}^{k-1}_{\mcl{I}_i}$. Let
$A=[a_{ij}]_{k \times k}\in M_{k \times k}(\R)$. Now we reduce the
matrix $A$ to an element in $SL_k(\R)$ to prove the claim in a step
by step manner.

Since each row generates the unit ideal using 
Lemma~\ref{lemma:Unital} we can right multiply $A$ by an
$SL_{k}(\R)\operatorname{-}$matrix so that $a_{11}$ element is a unit modulo
$\mcl{I}_1$. Now replace the first row by an equivalent row, where
$a_{11}=1$. Then we can transform the first row to
$e_1^k=(1,0,0,\ldots,0)$ using another $SL_k(\R)\operatorname{-}$matrix. Now we use
Theorem~\ref{theorem:RepHigherDimElement} to represent
appropriately the elements of the projective spaces by choosing the
map $\Gs=\Gs_{\mcl{F}}$ on the finitely generated monoid $\mathit{M}(\mcl{F}),$
where \equ{\mcl{F}=V(\mcl{I}_1)\cup V(\mcl{I}_2)\cup \ldots \cup
V(\mcl{I}_k).}

Let the second row be
\equ{[\Gs(\mcl{I}_{21})v_{21}:\Gs(\mcl{I}_{22})v_{22}:\ldots:\Gs(\mcl{I}_{2k})v_{2k}]}
We have \equ{\us{i=1}{\os{k}{\sum}}(\Gs(\mcl{I}_{2i})v_{2i})=\R,
\text{ and }
\mcl{I}_1+\us{i=2}{\os{k}{\sum}}(\Gs(\mcl{I}_{2i})v_{2i})=\R} by the
choice of the monoid. Hence we get
\equ{(\Gs(\mcl{I}_{21})v_{21})\mcl{I}_1+\us{i=2}{\os{k}{\sum}}(\Gs(\mcl{I}_{2i})v_{2i})=\R}
So there exists $i_1 \in \mcl{I}_1$ such that the vector
\equ{(\Gs(\mcl{I}_{21})v_{21}i_1,\Gs(\mcl{I}_{22})v_{22},\ldots,\Gs(\mcl{I}_{2k})v_{2k})}
is unital in $\R$. Now $\mcl{I}_2$ satisfies the $USC$. So by Proposition~\ref{prop:Unital} there exist
$s_1,s_3,\ldots,s_k \in \R$ such that the element
\equ{\Gs(\mcl{I}_{22})v_{22}+ \Gs(\mcl{I}_{21})v_{21}i_1s_1 +
\us{i=3}{\os{k}{\sum}}\Gs(\mcl{I}_{2i})v_{2i}s_i} is a unit modulo
$\mcl{I}_2$. The second summand in the above expression is in the
ideal $\mcl{I}_1$. Now we use a suitable column operation on $A$ to
transform $a_{22}$ to the above expression. This does not alter the
first row because it replaces the element $a_{12}$ by an element of
$\mcl{I}_1$. Hence we could replace the first row of $A$ back by
$e_1^k$. Now we have obtained $a_{22}$ a unit $\mod \mcl{I}_2$. We
can make this element $a_{22}=1$ exactly by replacing the second row
with another equivalent projective space element representative in
$\mbb{PF}^k_{\mcl{I}_2}$ however in the same equivalence class. Now
by applying suitable column operations we can transform the second
row to $e_2^k=(0,1,\ldots,0)$.

Inductively suppose we arrive at the $j^{th}\operatorname{-}$row for $j\leq k$. Let
the $j^{th}\operatorname{-}$row be given by
\equ{[\Gs(\mcl{I}_{j1})v_{j1}:\Gs(\mcl{I}_{j2})v_{j2}:\ldots:\Gs(\mcl{I}_{jk})v_{jk}]}
using again Theorem~\ref{theorem:RepHigherDimElement} with
respect to the same monoid map $\Gs$.

We have \equ{\us{i=1}{\os{k}{\sum}}(\Gs(\mcl{I}_{ji})v_{ji})=\R,
\text{ and } \mcl{I}_1\mcl{I}_2\ldots\mcl{I}_{j-1}
+\us{i=j}{\os{k}{\sum}}(\Gs(\mcl{I}_{ji})v_{ji})=\R} by the choice of
the monoid. Hence we get
\equ{\us{i=1}{\os{j-1}{\sum}}(\Gs(\mcl{I}_{ji})v_{ji})\mcl{I}_1\mcl{I}_2\ldots\mcl{I}_{j-1}+
\us{i=j}{\os{k}{\sum}}(\Gs(\mcl{I}_{ji})v_{ji})=\R}
So there exist $t_1,t_2,\ldots,t_{j-1} \in
\us{i=1}{\os{j-1}{\prod}}\mcl{I}_i$ such that the vector
\equ{(\Gs(\mcl{I}_{j1})v_{j1}t_1,\Gs(\mcl{I}_{j2})v_{j2}t_2,\ldots,
\Gs(\mcl{I}_{j(j-1)})v_{j(j-1)}t_{j-1},\Gs(\mcl{I}_{jj})v_{jj},\ldots,
\Gs(\mcl{I}_{jk})v_{jk})}
is unital in $R$. Now $\mcl{I}_j$ satisfies the $USC$. 
So by Proposition~\ref{prop:Unital} we make $a_{jj}$
element a unit $\mod \mcl{I}_j$ without actually changing the
previous $(j-1)$-rows as projective space elements because
$t_1,t_2,\ldots,t_{j-1}\in \us{i=1}{\os{j-1}{\bigcap}}\mcl{I}_i$.
Now we make the $a_{jj}=1$ exactly and then by applying (col. oper.) an
$SL_{k+1}(\R)$ matrix make the $j^{th}\operatorname{-}$row equal to
$e^k_j=(0,0,\ldots,0,1,0,\ldots,0)$.

We continue this procedure till $j=k$. We arrive at the identity
matrix. Hence the map $\gs_1$ is surjective and Claim~\ref{claim:SurjProj} follows.
\end{proof}
Continuing with the proof of Theorem~\ref{theorem:SurjSLkHighDimension}, we observe 
similarly the map $\gs_2$ is also surjective. This finishes the proof of Theorem~\ref{theorem:SurjSLkHighDimension}.
\end{proof}
Now we prove a corollary to Theorem~\ref{theorem:SurjSLkHighDimension}.
\begin{cor}
\label{cor:SurjGrassmannianHighDimension} Let $\R$ be a
commutative ring with unity.
\begin{enumerate}
\item Let $\R$ be a Dedekind type domain.
\item $\R$ has infinitely many maximal ideals.
\end{enumerate}
Let $\mathit{M}(\R)$ be the monoid generated by all maximal
ideals in $\R$. Let $\mcl{I}_1,\mcl{I}_2,\ldots,\mcl{I}_r\in
\mathit{M}(\R)$ be $r\operatorname{-}$ pairwise co-maximal proper ideals. Let $k \geq 2$ be
a positive integer. Consider for $r \leq k$
\equ{G_{r\times k}(\R)=\{A=[a_{ij}]_{r \times k} \in M_{r\times k}(\R) \mid
\text{ such that the } r\times r \text{ minors generate unit
ideal}\}.} Then the map \equ{\gt:G_{r\times k}(\R) \lra
\mbb{PF}^{k-1}_{\mcl{I}_1} \times \mbb{PF}^{k-1}_{\mcl{I}_2} \times
\ldots \times \mbb{PF}^{k-1}_{\mcl{I}_r}} given by \equa{\gt:(A) &=
([a_{11}:a_{12}:\ldots:a_{1k}],[a_{21}:a_{22}:\ldots:a_{2k}],\ldots,
[a_{r1}:a_{r2}:\ldots:a_{rk}])}
is surjective.
\end{cor}
\begin{proof}
Since the Dedekind type domain has infinitely many maximal ideals by
hypothesis, let $\mcl{I}_{r+1},\ldots,\mcl{I}_k \in \mathit{M}(\R)$
be pairwise co-maximal which are also co-maximal to each of
$\mcl{I}_1,\ldots,\mcl{I}_r$. Such ideals exist. Now using the main
Theorem~\ref{theorem:SurjSLkHighDimension} we conclude surjectivity
of this map $\gt$. Hence this
Corollary~\ref{cor:SurjGrassmannianHighDimension} also follows.
\end{proof}


\subsection{\bf{An example of a fixed point subgroup of $SL_k(\R)$, where surjectivity need 
not hold}}
\label{sec:FixedPointSubgroup}
In this section we give an example where Theorem~\ref{theorem:SurjSLkHighDimension} does not hold for a fixed
point subgroup. 
\begin{example}
\label{example:FixedPointSubgroup}
Let $\mbb{K}$ be an algebraically closed field. Let
$\R=\mbb{K}[z_1,z_2,\ldots,z_n]$. Consider the standard action of
$SL_2(\R)$ on $\R^2$. Let $G(\R)$ be the stabilizer subgroup of the
element $(1,1)^{tr}\in \R^2$ i.e. $G_2(\R)=\{A \in SL_2(\R)\mid
A.(1,1)^{tr}=(1,1)^{tr}\}$. Let $\mcl{M},\mcl{N}$ be two maximal
ideals in $\R$. Then the map \equ{G_2(\R) \lra \mbb{PF}^1_{\mcl{M}}
\times \mbb{PF}^1_{\mcl{N}}} is not surjective.

We observe that $G_2(\R)$ is also given as follows.
\equ{G_2(\R)=\{\mattwo{1+b}{-b}{b}{1-b}\mid b \in \R\}} So the image
of $G_2(\R)$ is exactly $\{([1+b:-b],[b:1-b])\mid b \in \R\} \subs
\mbb{PF}^1_{\mcl{M}} \times \mbb{PF}^1_{\mcl{N}}
=\mbb{PF}^1_{\mbb{K}} \times \mbb{PF}^1_{\mbb{K}} \subs
\mbb{PF}^3_{\mbb{K}}$. The image is precisely
\equ{([x_1:y_1],[x_2:y_2]) \in \mbb{PF}^1_{\mbb{K}} \times
\mbb{PF}^1_{\mbb{K}},\text{ where }(x_1+y_1)(x_2+y_2) \neq 0.} In
fact the image does not contain any element from the set
\equ{\bigg(\{[1:-1]\} \times \mbb{PF}^1_{\mbb{K}}\bigg) \bigcup
\bigg(\mbb{PF}^1_{\mbb{K}} \times \{[1:-1]\}\bigg)} which is a union
of two projective lines meeting at the point $([1:-1],[1:-1])$.
\end{example}


\section{\bf{A surjectivity theorem for the sum-product equation}}
\label{sec:SumProductEquation}
In this section we prove the following surjectivity theorem for the sum-product
equation.
\begin{theorem}
\label{theorem:SurjSumProductEquation} Let $\R$ be a commutative ring
with unity. Suppose $\R$ is a Dedekind type domain (refer to
Definition~\ref{defn:Good}). Let $\mathit{M}(\R)$ be the monoid
generated by maximal ideals in $\R$. Let
$\mcl{I}_1,\mcl{I}_2,\ldots,\mcl{I}_r\in \mathit{M}(\R)$ be $r\operatorname{-}$
pairwise co-maximal proper ideals. Let $r\geq 2,k \geq 2$ be two positive
integers. Consider \equan{SumProduct}{M_{r\times k}(\R)=\{A=[a_{ij}]_{r \times k}\mid
\us{j=1}{\os{k}{\sum}}\us{i=1}{\os{r}{\prod}}a_{ij}=1\}.} Then the
map \equ{\gl:M_{r\times k}(\R) \lra \mbb{PF}^{k-1}_{\mcl{I}_1} \times
\mbb{PF}^{k-1}_{\mcl{I}_2} \times \ldots \times
\mbb{PF}^{k-1}_{\mcl{I}_r}} given by \equ{\gl:(A) =
([a_{11}:a_{12}:\ldots:a_{1k}],[a_{21}:a_{22}:\ldots:a_{2k}],\ldots,
[a_{r1}:a_{r2}:\ldots:a_{rk}])}
is surjective.
\end{theorem}
\begin{proof}
We have $r \geq 2$. Let us prove this by induction on $r$. First we
prove for $r=2$. Let \equ{([x_1:\ldots:x_k],[y_1:\ldots:y_k]) \in
\mbb{PF}^{k-1}_{\mcl{I}_1} \times \mbb{PF}^{k-1}_{\mcl{I}_2}.}
Suppose there exists
\equ{([x^0_1:\ldots:x^0_k],[y^0_1:\ldots:y^0_k])=([x_1:\ldots:x_k],[y_1:\ldots:y_k])\in
\mbb{PF}^{k-1}_{\mcl{I}_1} \times \mbb{PF}^{k-1}_{\mcl{I}_2}} such
that $\us{j=1}{\os{k}{\sum}} x^0_jy^0_j = 1 + i_2,$ where $i_2 \in
\mcl{I}_2$. Let $\us{j=1}{\os{k}{\sum}} x^0_jz^0_j=1$ because we
have $\us{j=1}{\os{k}{\sum}} (x^0_j)=\R$. By choosing
$u_j=x^0_j,v_j=y^0_j-z_j^0i_2$ we have $\us{j=1}{\os{k}{\sum}}
u_jv_j=1$ and
\equ{([u_1:\ldots:u_k],[v_1:\ldots:v_k])=([x_1:\ldots:x_k],[y_1:\ldots:y_k])\in
\mbb{PF}^{k-1}_{\mcl{I}_1} \times \mbb{PF}^{k-1}_{\mcl{I}_2}} So it
is enough to prove that there exists
$([x^0_1:\ldots:x^0_k],[y^0_1:\ldots:y^0_k])=([x_1:\ldots:x_k],[y_1:\ldots:y_k])\in
\mbb{PF}^{k-1}_{\mcl{I}_1} \times \mbb{PF}^{k-1}_{\mcl{I}_2}$ such
that $\us{j=1}{\os{k}{\sum}} x^0_jy^0_j \equiv 1 \mod \mcl{I}_2$.
Since $\mcl{I}_1+\mcl{I}_2=\R$, let $a \in \mcl{I}_1,b\in \mcl{I}_2$
such that $a+b=1-\us{i=1}{\os{k}{\sum}}x_iy_i$. Now there exist
$w_i\in \R$ such that $\us{i=1}{\os{k}{\sum}}w_iy_i=1$ because
$\us{i=1}{\os{k}{\sum}}(y_i)=\R$. Hence $\us{i=1}{\os{k}{\sum}}
x_iy_i+aw_iy_i=1-b \equiv 1 \mod \mcl{I}_2$. Now we have
\equ{\us{j=1}{\os{k}{\sum}}(x_j+aw_j)+\mcl{I}_2=\R,
\us{j=1}{\os{k}{\sum}}(x_j+aw_j)+\mcl{I}_1=\R.} Hence
$\us{j=1}{\os{k}{\sum}}(x_j+aw_j)+\mcl{I}_1\mcl{I}_2=\R$. So using
Proposition~\ref{prop:UnitalDD} we conclude that there exist
$t_1,t_2,\ldots,t_k\in \mcl{I}_1\mcl{I}_2$ such that
$\us{j=1}{\os{k}{\sum}}(x_j+aw_j+t_j)=\R$ and
\equ{\us{j=1}{\os{k}{\sum}}
x_jy_j+aw_jy_j+t_jy_j=1-b+\us{j=1}{\os{k}{\sum}}t_jy_j \equiv 1 \mod
\mcl{I}_2} So choosing $x_j^0=x_j+aw_j+t_j,y_j^0=y_j$ and plugging it back for $u_j,v_j$ proves
this Theorem~\ref{theorem:SurjSumProductEquation} for the case when $r=2$. In fact we have 
\equ{u_j\equiv x_j \mod \mcl{I}_1, v_j \equiv y_j \mod \mcl{I}_2,\us{i=1}{\os{k}{\sum}}u_iv_i=1.} 

Now we prove for any positive integer $r>2$. Let
\equ{\mcl{F}=V(\mcl{I}_1) \cup V(\mcl{I}_2) \cup \ldots \cup
V(\mcl{I}_r).} It is a finite set because $V(\mcl{I}_i)$ is a finite set for every $1\leq i\leq r$. 
Let $\Gs=\Gs_{\mcl{F}}:\mathit{M}(\mcl{F}) \lra \R$ be
the nowhere zero choice multiplicative monoid map using
Theorem~\ref{theorem:Comaximality}. We choose this map $\Gs=\Gs_{\mcl{F}}$ in Theorem~\ref{theorem:RepHigherDimElement}
to obtain a representative element
\equa{&([\Gs(\mcl{J}_{11})v_{11}:\Gs(\mcl{J}_{12})v_{12}:\ldots:\Gs(\mcl{J}_{1k})v_{1k}],
[\Gs(\mcl{J}_{21})v_{21}:\Gs(\mcl{J}_{22})v_{22}:\ldots:
\Gs(\mcl{J}_{2k})v_{2k}],\ldots,\\
&[\Gs(\mcl{J}_{r1})v_{r1}:\Gs(\mcl{J}_{r2})v_{r2}:\ldots:\Gs(\mcl{J}_{rk})v_{rk}])
\in \us{i=1}{\os{r}{\prod}}\mbb{PF}^{k-1}_{\mcl{I}_i}.} Let
$\mcl{I}=\us{i=1}{\os{r}{\prod}} \mcl{I}_i$. We note that
$(v_{ij})+\mcl{I}=\R$ for every $(i,j)\in \{1,2,\ldots,r\} \times
\{1,2,\ldots,k\}$. We replace $v_{ij}$ by $w_{ij}\in v_{ij}+\mcl{I}$
such that the following two properties hold.

\begin{itemize}
\item For every $(i,j) \neq (e,f) \in \{1,2,\ldots,r\} \times
\{1,2,\ldots,k\}$ the sets of maximal ideals containing $w_{ij}$ and
$w_{ef}$ are disjoint i.e. $V(w_{ij}) \cap V(w_{ef})=\es$.

\item For every $(i,j),(e,f) \in \{1,2,\ldots,r\} \times
\{1,2,\ldots,k\}$ the sets of maximal ideals containing $w_{ij}$ and
$\Gs(\mcl{J}_{ef})$ are disjoint i.e. $V(w_{ij}) \cap
V(\Gs(\mcl{J}_{ef})) = \es$.

\end{itemize}

This can be done using Lemma~\ref{lemma:FundLemma} on
arithmetic progressions for Dedekind type domains. This immediately
implies that for each $i$ we have a well defined element
representing the same element
\equ{[\Gs(\mcl{J}_{i1})w_{i1}:\Gs(\mcl{J}_{i2})w_{i2}:\ldots:\Gs(\mcl{J}_{ik})w_{ik}]=
[\Gs(\mcl{J}_{i1})v_{i1}:\Gs(\mcl{J}_{i2})v_{i2}:\ldots:\Gs(\mcl{J}_{ik})v_{ik}]
\in \mbb{PF}^{k-1}_{\mcl{I}_i}}

We observe that any maximal ideal containing the coordinates
$\Gs(\mcl{J}_{ij})w_{ij},1 \leq j \leq k$  contains all
$\Gs(\mcl{J}_{ij})$ for $1 \leq j \leq k$ and hence has to be the unit
ideal which is a contradiction.

Now for a fixed $1 \leq j \leq k$ we observe that the maximal ideals
containing $\Gs(\mcl{J}_{ij})$ outside $V(\mcl{I}_i)$ distinct for
$1 \leq i \leq r$ using Observation~\ref{obs:DistinctMax} as we have $\mcl{J}_{ij} \sups
\mcl{I}_i$ for $1 \leq i \leq r$ with $\mcl{I}_i$ being mutually
co-maximal. We have for $1 \leq j \leq k$
\equ{\us{i=2}{\os{r}{\prod}}\Gs(\mcl{J}_{ij})=\Gs(\us{i=2}{\os{r}{\prod}}\mcl{J}_{ij}).}

\begin{claim}
The set \equ{\{\us{i=2}{\os{r}{\prod}}\Gs(\mcl{J}_{ij})\mid j=1,\ldots,k\}}
is unital.
\end{claim}
\begin{proof}[Proof of Claim]
We have $r>2$. Consider a maximal ideal $\mcl{M}$ containing the set
$\{\us{i=2}{\os{r}{\prod}}\Gs(\mcl{J}_{ij})\mid j=1,\ldots,k\}$.
Then for some $1\leq j\leq k, \mcl{M}$ contains one of the factors $\Gs(\mcl{J}_{ij})$ for
every $2 \leq i \leq r$. If $\Gs(\mcl{J}_{i_1j_1}),\Gs(\mcl{J}_{i_2j_2})\in \mcl{M}$ with $2\leq i_1\neq i_2 \leq r$  
then $\Gs(\mcl{N}_1^{t}),\Gs(\mcl{N}_2^{s})\in \mcl{M}$ for some $t,s\geq 0$ and distinct maximal ideals $\mcl{N}_1,\mcl{N}_2\in \mcl{F}$
with $\mcl{N}_1\sups \mcl{I}_{i_1},\mcl{N}_2 \sups \mcl{I}_{i_2}$.
This contradicts Observation~\ref{obs:DistinctMax}.
Now if for some $2\leq i\leq r, \Gs(\mcl{J}_{ij})\in \mcl{M}$ for all $1\leq j\leq k$ then this is a contradiction because  
$\us{j=1}{\os{k}{\sum}}(\Gs(\mcl{J}_{ij}))=\R$. So the set
\equ{\{\us{i=2}{\os{r}{\prod}}\Gs(\mcl{J}_{ij})\mid j=1,\ldots,k\}}
is unital. This proves the claim.
\end{proof}
Continuing with the proof of the theorem, similarly now we observe that the set
\equ{\{\us{i=2}{\os{r}{\prod}}\Gs(\mcl{J}_{ij})w_{ij}=
\Gs(\us{i=2}{\os{r}{\prod}}\mcl{J}_{ij})w_{ij}\mid
j=1,\ldots,k\}} is also unital.

Now consider the element
\equ{[\us{i=2}{\os{r}{\prod}}\Gs(\mcl{J}_{i1})w_{i1}:
\us{i=2}{\os{r}{\prod}}\Gs(\mcl{J}_{i2})w_{i2}:\ldots:
\us{i=2}{\os{r}{\prod}}\Gs(\mcl{J}_{ik})w_{ik}]
\in
\mbb{PF}^{k-1}_{\bigg(\big(\us{i=2,j=1}{\os{r,k}{\prod}}\Gs(\mcl{J}_{ij})w_{ij}\big)\mcl{I}_2
\mcl{I}_3\ldots\mcl{I}_r\bigg)}.}

Now we reduce to the case when $r=2$ and apply the case $r=2$ for the above element
in
\equ{\mbb{PF}^{k-1}_{\bigg(\big(\us{i=2,j=1}{\os{r,k}{\prod}}\Gs(\mcl{J}_{ij})w_{ij}\big)
\mcl{I}_2\mcl{I}_3\ldots\mcl{I}_r\bigg)}}
and the element
$[\Gs(\mcl{J}_{11})w_{11}:\Gs(\mcl{J}_{12})w_{12}:\ldots:\Gs(\mcl{J}_{1k})w_{1k}]
\in \mbb{PF}^{k-1}_{\mcl{I}_1}$. We note that the two ideals
\equ{\bigg(\big(\us{i=2,j=1}{\os{r,k}{\prod}}\Gs(\mcl{J}_{ij})w_{ij}\big)\mcl{I}_2\mcl{I}_3
\ldots\mcl{I}_r\bigg),\mcl{I}_1}
are co-maximal. Now there exist elements $b_{1j}: 1\leq j \leq k$ with $b_{1j}
\equiv \Gs(\mcl{J}_{1j})w_{1j} \mod \mcl{I}_1$ and
\equ{[b_{11}:b_{12}:\ldots:b_{1k}]=[\Gs(\mcl{J}_{11})w_{11}:\Gs(\mcl{J}_{12})w_{12}\ldots:
\Gs(\mcl{J}_{1k})w_{1k}]
\in \mbb{PF}^{k-1}_{\mcl{I}_1}}  and there exist $t_1,t_2,\ldots,t_k
\in \mcl{I}_2\mcl{I}_3\ldots\mcl{I}_r$ such that
\equ{\us{l=1}{\os{k}{\sum}}b_{1l}\bigg(\us{i=2}{\os{r}{\prod}}\Gs(\mcl{J}_{il})w_{il}+
t_l\us{i=2,j=1}{\os{r,k}{\prod}}\Gs(\mcl{J}_{ij})w_{ij}\bigg)=1=\us{l=1}{\os{k}{\sum}}b_{1l}\bigg(\us{i=2}{\os{r}{\prod}}\Gs(\mcl{J}_{il})w_{il}\bigg)\bigg(1+t_l\us{i=2,j=1,j\neq l}{\os{r,k}{\prod}}\Gs(\mcl{J}_{ij})w_{ij}\bigg).}

Now consider the same element with these representatives
\equ{([b_{11}:b_{12}:\ldots:b_{1k}],[b_{21}:b_{22}:\ldots:b_{2k}],\ldots,[b_{r1}:b_{r2}:
\ldots:b_{rk}])
\in \mbb{PF}^{k-1}_{\mcl{I}_1} \times \mbb{PF}^{k-1}_{\mcl{I}_2}\times
\ldots \times \mbb{PF}^{k-1}_{\mcl{I}_r},} where for $r>l>1$ we have
$b_{ij}=\Gs(\mcl{J}_{ij})w_{ij}$ and for $l=r$ we have
\equ{b_{rl}=\Gs(\mcl{J}_{rl})w_{rl}\bigg(1+t_l\us{i=2,j=1,j\neq l}{\os{r,k}{\prod}}\Gs(\mcl{J}_{ij})w_{ij}\bigg)=\Gs(\mcl{J}_{rl})w_{rl}+t_l\bigg(\frac{\us{i=2,j=1}{\os{r,k}{\prod}}
\Gs(\mcl{J}_{ij})w_{ij}}{\us{i=2}{\os{r-1}{\prod}}\Gs(\mcl{J}_{il})w_{il}}\bigg).}
Then we observe that
\equ{\us{j=1}{\os{k}{\sum}}\us{i=1}{\os{r}{\prod}}b_{ij}=1.}

The map $\gl$ is surjective and Theorem~\ref{theorem:SurjSumProductEquation} follows for
any $r>2$.
\end{proof}
\begin{note}
For $r=1$ Theorem~\ref{theorem:SurjSumProductEquation} is not true. 
Choose $\R=\Z$. $\mcl{I}_1=p\Z$. The point \equ{[1:-1]\in \mbb{PF}^1_{\mcl{I}_1}=\mbb{PF}^1_{p}} is
not in the image of $M_{1\times 2}(\Z)$.
\end{note}


\section{\bf{Appendix}}
\label{sec:Appendix}
~\\
In this section we prove that a noetherian ring $\mcl{R}$ which satisfies properties $(1),(2),(3)$ in Definition~\ref{defn:Good} is a Dedekind domain.
We state the theorem as:
\begin{theorem}
\label{theorem:DD}
Let $\mcl{R}$ be a noetherian ring which satisfies properties $(1),(2),(3)$ in Definition~\ref{defn:Good}. Then $\mcl{R}$ is a Dedekind domain and hence 
$\mcl{R}$ satisfies property $(4)$ in Definition~\ref{defn:Good}.
\end{theorem}
\begin{proof}
First we prove that $\mcl{R}$ is a domain with the following claim.
\begin{claim}
\label{claim:ID}
With the hypothesis in Theorem~\ref{theorem:DD} let $\mcl{M}$ be a maximal ideal. Let $\mcl{S}_{\mcl{M}^i}=\mcl{M}^i\bs \mcl{M}^{i+1},i\geq 0$.
Then for every $0\leq i\leq j$, there exist $s_i\in \mcl{S}_{\mcl{M}^i},s_j\in \mcl{S}_{\mcl{M}^j}$ such that $s_is_j\in \mcl{S}_{\mcl{M}^{i+j}}$ and hence non-zero.
\end{claim}
\begin{proof}[Proof of Claim]
Suppose not then we have 
\equ{\mcl{M}^{i+j}=(\mcl{M}^i).(\mcl{M}^j)=\sum \big(\mcl{S}_{\mcl{M}^i}\bigsqcup \mcl{M}^{i+1}\big).\big(\mcl{S}_{\mcl{M}^j}\bigsqcup \mcl{M}^{j+1}\big) \subs \mcl{M}^{i+j+1}}
which is a contradiction to property $(1)$. Hence the claim follows.
\end{proof}
Let $a,b\in \mcl{R}^{*}$. Then there exist $0\leq i,j$ such that $a\in \mcl{S}_{\mcl{M}^i},b\in \mcl{S}_{\mcl{M}^j}$ using property $(2)$.
Using Claim~\ref{claim:ID} there exist $s_i\in \mcl{S}_{\mcl{M}^i},s_j\in \mcl{S}_{\mcl{M}^j}$ such that $s_is_j\in \mcl{S}_{\mcl{M}^{i+j}}$. Using property $(3)$ we have 
$a=ss_i+a',b=ts_j+b'$ where $s,t\in \mcl{R}\bs \mcl{M}=\mcl{S}_{\mcl{M}^0},a'\in \mcl{M}^{i+1},b'\in \mcl{M}^{j+1}$. Hence we obtain 
\equ{ab=sts_is_j+c \text{ where }c\in \mcl{M}^{i+j+1} \text{ and } s,t \text{ are units modulo }\mcl{M}^{k} \text{ for any }k\geq 1.}
Hence $sts_is_j\in \mcl{S}_{\mcl{M}^{i+j}}$ and therefore $ab\neq 0$.
 
Secondly we assume that $\mcl{R}$ is a noetherian local ring with unique maximal ideal $\mcl{M}$ so that it is enough to prove that $\mcl{R}$ is a discrete valuation ring.
For this purpose we prove that $\widehat{\mcl{R}}$ is a discrete valuation ring where $\widehat{\mcl{R}}$ is the $\mcl{M}\operatorname{-}$adic completion of $\mcl{R}$. 
\begin{claim}
With the hypothesis in Theorem~\ref{theorem:DD} if $\R$ is a noetherian local ring then $\widehat{\R}$ is a discrete valuation ring.
\end{claim}
\begin{proof}[Proof of Claim]
Using a similar argument in Claim~\ref{claim:ID} we conclude that for any $t\in \mcl{M}\bs \mcl{M}^2, t^k \in \mcl{M}^k\bs \mcl{M}^{k+1}$ for any $k\geq 1$. 
Fix one such element $t$. Now every element $x\in \big(\mcl{M}^n \bs \mcl{M}^{n+1}\big)\subs \mcl{R}$ has a power series expansion as 
\equ{x=s_nt^n+s_{n+1}t^n\ldots \text{ with }s_i\in \big(\mcl{R}\bs \mcl{M}\big)\cup\{0\}, i\geq n,s_n\neq 0}
which converges in $\widehat{\mcl{R}}$. Using M.~F.~Atiyah and I.~G.~Macdonald~\cite{AM}, proposition $10.15$ on page $109$, 
we conclude that $\widehat{\mcl{M}}^n=\widehat{\mcl{R}}\mcl{M}^n=(t^n)$ and $\frac{\widehat{\mcl{M}}^n}{\widehat{\mcl{M}}^{n+1}}\cong \frac{\mcl{M}^n}{\mcl{M}^{n+1}}$. On the quotient field $\mbb{K}$
of $\widehat{\mcl{R}}$ we have a well defined valuation such that $\widehat{R}$ is a discrete valuation ring of $\mbb{K}$. This proves the claim.
\end{proof}
So $\widehat{\mcl{R}}$ is a noetherian local domain
of dimension one. Hence $\mcl{R}$ is also noetherian local domain of dimension one because dim $\mcl{R}=$ dim $\widehat{\mcl{R}}$ which implies $\mcl{R}$ is a discrete valuation ring
using M.~F.~Atiyah and I.~G.~Macdonald~\cite{AM}, proposition $9.2$ on page $94$. 

So in general the ring $\mcl{R}$ is a noetherian domain of dimension one whose local rings are discrete valuation rings. 
Now using M.~F.~Atiyah and I.~G.~Macdonald~\cite{AM}, Theorem $9.3$ on page $95$ we conclude that $\mcl{R}$ is a Dedekind domain. 

Now property $(4)$ in Definition~\ref{defn:Good} follows because of the existence of ideal factorization as a product of maximal ideals. This proves the theorem. 
\end{proof}


\section{\bf{Acknowledgments}}
The author thanks the referee profusely for going through the paper and suggesting improvements. The author also thanks his mentor Prof. B. Sury for his support,
encouragement and useful comments. The author is supported by Indian Statistical Institute (ISI) grant as a visiting scientist at ISI Bangalore, India.  


\bibliographystyle{abbrv}
\def\cprime{$'$}

\end{document}